\documentclass[a4paper,12pt,reqno]{amsart}
\usepackage{amssymb}
\usepackage{amsthm}
\usepackage{bm}
\usepackage{latexsym}
\usepackage{amsmath}
\usepackage{eufrak}
\usepackage{mathrsfs}
\usepackage{caption}
\usepackage{amscd}
\usepackage[all,cmtip]{xy}
\usepackage[usenames]{color}
\usepackage{graphicx}
\usepackage{color}
\usepackage{amscd}
\usepackage{float}
\usepackage{graphics}
\usepackage{tikz}
\usepackage{tikz-cd}
\usepackage{comment}
\usepackage{xspace}
\usepackage{mathtools}

\usepackage{amssymb}
\usepackage{amsthm}
\usepackage{graphicx}
\usepackage{subfig}
\usepackage{enumerate}
\usepackage{hyperref}

\usetikzlibrary{patterns,decorations.pathreplacing}
\usepackage{marginnote}

\theoremstyle{plain}

\newtheorem{theorem}{Theorem}[section]
\newtheorem{proposition}[theorem]{Proposition}

\newtheorem{corollary}[theorem]{Corollary}

\theoremstyle{definition}

\newcommand{\appsection}[1]{\let\oldthesection\thesection
\renewcommand{\thesection}{Appendix \oldthesection}
\section{#1}\let\thesection\oldthesection}

\newtheorem{definition}[theorem]{Definition}

\theoremstyle{remark}

\newtheorem{remark}[theorem]{Remark}

\def\D{{\mathbb{D}}}

\def\Z{{\mathbb{Z}}}

\def\Q{{\mathbb{Q}}}
\def\C{{\mathbb{C}}}
\def\P{{\mathbb{P}}}

\def\O{{\mathcal{O}}}

\def\X{{\mathcal{X}}}

\def\AR{{\mathcal{A}}}

\newlength\templength

\DeclareMathOperator{\exc}{exc}

\begin{document}
\title{Rational configurations in K3 surfaces and simply-connected $p_g=1$ surfaces for $K^2=1,2,3,4,5,6,7,8,9$}
\dedicatory{Dedicated to Fabrizio Catanese on the occasion of his 71st birthday}

\author[Javier Reyes]{Javier Reyes}
\email{jereyes4@uc.cl}
\address{Facultad de Matem\'aticas, Pontificia Universidad Cat\'olica de Chile, Campus San Joaqu\'in, Avenida Vicu\~na Mackenna 4860, Santiago, Chile.}

\author[Giancarlo Urz\'ua]{Giancarlo Urz\'ua}
\email{urzua@mat.uc.cl}
\address{Facultad de Matem\'aticas, Pontificia Universidad Cat\'olica de Chile, Campus San Joaqu\'in, Avenida Vicu\~na Mackenna 4860, Santiago, Chile.}


\begin{abstract}
We prove the existence of $(20-2K^2)$-dimensional families of simply-connected surfaces with ample canonical class, $p_g=1$, and $1 \leq K^2 \leq 9$, and we study the relation with configurations of rational curves in K3 surfaces via $\mathbb Q$-Gorenstein smoothings. Our surfaces with $K^2=7$ and $K^2=9$ are the first surfaces known in the literature, together with the existence of a $4$-dimensional family for $K^2=8$.
\end{abstract}



\maketitle

\tableofcontents

\section{Introduction} \label{s1}

$\Q$-Gorenstein smoothings allow us to relate configurations of rational curves in rational surfaces with families of $p_g=0$ surfaces. They include Enriques surfaces (starting with a Coble surface for example), properly elliptic surfaces (in particular Dolgachev surfaces), and surfaces of general type (in particular simply-connected, although non-trivial fundamental groups can appear too, as shown for Godeaux surfaces in \cite{CU18,DRU20}). In fact, the only known examples of simply-connected minimal surfaces of general type with $p_g=0$ and $K^2=2,3,4$ have been constructed via $\Q$-Gorenstein smoothings (see \cite{LP07}, \cite{PPS09a}, \cite{PPS09b} for the first examples; see e.g. \cite{SU16} for several other examples). 

Similarly, simply-connected $p_g=1$ surfaces are related to configurations of rational curves in K3 surfaces. In \cite{PPS13} and via $\Q$-Gorenstein smoothings, J. Park, H. Park, and D. Shin proved the existence of $(20-2K^2)$-dimensional families of simply-connected minimal surfaces of general type with $p_g=1$ and $K^2=3,4,5,6$, and also the existence of such surfaces with $K^2=8$.

We recall that $p_g=1$ surfaces were studied in the eighties for their relation with counterexamples to Torelli type of theorems (see \cite{Ky77}, \cite{C80} and \cite{T80, T81}). For $K^2=1$, Catanese proved \cite{C80} that the moduli space is an irreducible rational variety of dimension $18$.  (Beauville \cite{B14} showed, using that the image of the period map is open \cite{C79}, that in this moduli space the surfaces with maximal Picard number are dense.) These surfaces are simply-connected, and they can be realized as weighted complete intersections $(6,6)$ in $\P(1,2,2,3,3)$. (For similar results in the situation of Gorenstein stable surfaces see \cite{FPR17}.) 

Less is known in the case of $K^2=2$. It can be shown that the torsion of the Picard group is either $\{0\}$ or $\Z/2$. For the latter, Catanese and Debarre \cite{CD80} proved that their moduli space is a rational irreducible variety of dimension $16$, and that these surfaces have $\pi_1=\Z/2$. They also worked out various properties for torsion zero surfaces, including an irreducible $16$ dimensional component of the moduli space, but it is still an open problem whether this is the only component for the moduli space of torsion zero surfaces. If that were the case, then all these surfaces would be simply-connected. We note that these particular surfaces are relevant for the description of the moduli space of $\Z/2$-Godeaux surfaces, which was recently proved \cite{DR20} to be a rational irreducible variety of dimension $8$, finishing with the previous attempts e.g. \cite{CD80,CCO94,Co16}. For $K^2=3$ much less is known, the torsion of the Picard group has order at most $4$, and the ones with torsion $\Z/3$ are described in \cite{M03}. Up until now, the article \cite{PPS13} provides the only other examples of simply-connected $p_g=1$ surfaces of general type.  

In the present paper, we study the relation between configurations of rational curves in K3 surfaces and simply-connected $p_g=1$ surfaces via $\Q$-Gorenstein smoothings, and we prove the existence of $(20-2K^2)$-dimensional families for $K^2=7,8,9$. This finishes with all possible values of $K^2$ in the case of unobstructed $\Q$-Gorenstein smoothings, where surfaces with higher $K^2$ are harder to find. A particular characteristic of our constructions is that we use configurations from one fixed simple normal crossings (SNC) configuration formed by smooth rational curves (see Section \ref{s4}).

\begin{theorem}
There are normal projective surfaces $X$ with ample canonical class, and the following properties:

\begin{itemize}
\item[(1)] They have only cyclic quotient singularities of type $\frac{1}{n^2}(1,na-1)$ with gcd$(n,a)=1$ (i.e. Wahl singularities), and rational double points. The minimal resolution of $X$ is a suitable composition blow-ups from a K3 surface.

\item[(2)] $q(X)=0$, $p_g(X)=1$, and $K_X^2=1,2,3,4,5,6,7,8,9$.

\item[(3)] The surface $X$ has no-local-to-global obstructions to deform, and in particular it represents a family of dimension $20-2K_X^2$ in the Koll\'ar--Shepherd-Barron--Alexeev (KSBA) compactification of the moduli space of surfaces of general type.
\end{itemize}
\label{main}
\end{theorem}

For each of these surfaces, we are able to compute the fundamental group of a $\Q$-Gorenstein smoothing, and so we obtain the following.

\begin{corollary}
For each $k \in \{ 1,2,3,4,5,6,7,8,9 \}$, there are simply-connected nonsingular surfaces with ample canonical class, $K^2=k$, and $p_g=1$, whose deformation space has dimension $20-2k$. In particular, they are exotic minimal symplectic $3 \C \P^2 \# (19-k) \overline{\C \P^2}$.
\end{corollary}

Each Wahl singularity in $X$ defines a divisor in the KSBA compactification of the moduli space of surfaces of general type. The general surface $Y$ in that divisor has one Wahl singularity. Let us suppose that the surface $Y$ can also be constructed from a nodal configuration $\AR$ of $(-2)$-curves in a K3 surface $S$. (This is a situation that happens often.) Thus $\AR$ defines a sublattice $M$ in Pic$(S)$ and, by Dolgachev \cite[Proposition 2.1]{D96}, it defines a local moduli space $S_M$ of isomorphism classes of $M$-polarized K3 surfaces. It can be checked that the dimension of $S_M$ matches the dimension of the corresponding divisor (see Remark \ref{igor}). It may also be the case that the same configuration defines two distinct divisors in the same moduli space. It is not clear to us the direct relation between configurations and Wahl singularities. One aim of this paper is to elaborate on some connections between bounds of lengths of Wahl singularities and logarithmic Chern numbers of configurations.

Another particular property is that the Picard number of the K3 surface must be at least $r=P+2K^2 \geq 3$, where $P$ is the number of Wahl singularities in $X$ and $r$ is the number of curves in the configuration. It is well-known that such K3 surfaces contain a finite number of $(-2)$-curves if and only if they have a finite automorphism group (which is equivalent to being a Mori dream space), and they were all classified by their Picard lattices, cf. \cite{Ro20}. It turns out that to construct our surfaces we need configurations of $(-2)$-curves in surfaces with infinitely many $(-2)$-curves, in particular, we use elliptic extremal K3 surfaces \cite{Ye99,SZ01}. In this paper, we actually obtain our results from just one of them. We work out all details for specific examples for each $K^2$. At the end we provide some others, but the complete world of them seems to be huge, depending on the value of $K^2$.

In Section \ref{s2}, we describe general $\Q$-Gorenstein degenerations of regular $p_g=1$ surfaces and general bounds for lengths of Wahl singularities in the special fiber. To exemplify with the simplest case, we classify all possibilities for $K^2=1$ and one Wahl singularity (Theorem \ref{OneWahlK=1}), realizing most of them after showing a general tool for obstructions from nodal rational configurations in K3 surfaces (Proposition \ref{obstruction}). In Section \ref{s3}, we study the geography of SNC configurations of rational curves in K3 surfaces in relation to the surfaces we want to find (Proposition \ref{chern}). We also review various criteria to check the positivity of $K$, to compute the fundamental group of general fibers, and to obtain unobstructedness in our case. In Section \ref{s4}, we prove Theorem \ref{main} and its corollary, showing also several examples at the end with particular properties, among them we show some T-singularities $\frac{1}{2n^2}(1,2na-1)$ and wormholes \cite{UV21}. In a sequel paper, we will work out new exotic surfaces with higher $K^2$, in the cases of simply-connected $p_g=1$ and also $p_g=0$ surfaces.

\subsubsection*{Notation and conventions}

\noindent

\begin{itemize}

\item A simple normal crossings (SNC) divisor (or configuration) is a finite collection of nonsingular curves on a nonsingular surface so that its union has only nodes as singularities.

\item A $(-m)$-curve $C$ on a nonsingular surface is a $C \simeq \P^1$ with $C^2 = -m$.

\item The Kodaira dimension of a nonsingular surface $S$ is denoted by $\kappa(S)$.

\item For a projective surface $X$: $q(X)=h^1(X,\O_X)$ and $p_g(X)=h^2(X,\O_X)$.

\item A nonsingular projective surface with an elliptic fibration over $\mathbb P^1$ which has one fiber of multiplicity $n$ and $p_g=1$ is denoted by $D_{n}$.

\item For a birational morphism $\phi \colon \tilde X \to X$, the exceptional loci is denoted by $\exc(\phi)$.

\end{itemize}

\subsubsection*{Acknowledgments}
The first-named author was funded by the ANID scholarship 22201484. The second-named author was supported by the FONDECYT regular grant 1190066. This work was also supported by ANID - Millennium Science Initiative Program - grant NCN$17\textunderscore \,059$ ``Millenium Nucleus Center for the Discovery of Structures in Complex Data".

\section{Wahl degenerations of $p_g=1$ surfaces} \label{s2}

\textit{Cyclic quotient singularities} will be denoted by $\frac{1}{m}(1,q)$, meaning the isomorphism class of the germ at the origin of the quotient of $\C^2$ by the action $(x,y)\mapsto (\mu x, \mu^q y)$, where $\mu$ is a primitive $m$-th root of $1$, and $q$ is an integer with $0<q<m$ and gcd$(q,m)=1$. We refer to \cite[2.1]{U16b} or \cite[III Sect. 5]{BHPV04} for basic facts. Their minimal resolution is achieved by replacing the singular point by a chain of smooth rational curves $\{E_1,\ldots,E_{\ell}\}$ with $E_i^2=-b_i<-1$, where one reads the $b_i$ from the \textit{Hirzebruch-Jung continued fraction} $$ \frac{m}{q} = b_1 - \frac{1}{b_2 - \frac{1}{\ddots - \frac{1}{b_{\ell}}}} =: [b_1, \ldots ,b_{\ell}].$$ The \textit{length} of $\frac{1}{m}(1,q)$ is $\ell$. If $\sigma \colon \tilde{Y} \to Y=\frac{1}{m}(1,q)$ is the minimal resolution, then the \textit{discrepancies} of the $E_i$ are the coefficients of the $E_i$ in the difference $K_{\tilde{Y}}-\sigma^*(K_Y)$. Discrepancies are rational numbers in the interval $(-1,0]$.

Let $\D$ be the smooth analytic germ of a curve. For a surface $X$ with only quotient singularities, a deformation $(X \subset \X) \to (0 \in \D)$ of $X$ is called a \textit{smoothing} if its general fiber is smooth. It is a \textit{$\Q$-Gorenstein smoothing} if in addition $K_{\X}$ is $\Q$-Cartier. A \textit{T-singularity} is a quotient singularity which admits a $\Q$-Gorenstein smoothing \cite[Def. 3.7]{KSB88}. As shown in \cite[Prop. 3.10]{KSB88}, T-singularities are either rational double points (RDP) or cyclic quotient singularities of type $\frac{1}{dn^2}(1,dna-1)$ with $0<a<n$ and gcd$(a,n)=1$, where the dimension of its $\Q$-Gorenstein deformation space is $d$. A \textit{Wahl singularity} is a non RDP T-singularity with $d=1$ \cite{Wahl81}. Equivalently, a Wahl singularity is a cyclic quotient singularity which admits a smoothing whose Milnor number is equal to zero. In \cite{Wahl81}, Wahl describes how to identify singularities $\frac{1}{n^2}(1,na-1)$ with gcd$(a,n)=1$ from its minimal resolution (see \cite[Prop. 3.11]{KSB88}), which we call \textit{Wahl chain}: 
\textbf{(i)} $[4]$ is a Wahl chain, \textbf{(ii)} if $[b_1,\ldots,b_{\ell}]$ is a Wahl chain, then $[b_1+1,\ldots,b_{\ell},2]$ and $[2,b_1,\ldots,b_{\ell}+1]$ are both Wahl chains, and \textbf{(iii)} any Wahl chain is obtained by starting with (i) and iterating (ii).

The discrepancies of Wahl singularities can be easily computed via a Fibonacci sort of algorithm (see e.g. \cite[2.3]{UV21}), which also implies the following relation between the \textit{index} $n$ of the singularity and its length $\ell$:   If $\mathcal F_i$ is the i-th Fibonacci number defined as $\mathcal F_{-1}=\mathcal F_{0}=1$ and $\mathcal F_i = \mathcal F_{i-1} + \mathcal F_{i-2}$, then $n \leq \mathcal F_{\ell}$. Equality holds if and only if the Wahl chain is $[3,\ldots,3,5,3,\ldots,3,2]$ or $[2,3,\ldots,3,5,3,\ldots,3]$.

We will abbreviate the information of a normal projective surface $X$ with only Wahl singularities and a $\Q$-Gorenstein smoothing of $X$ over $\D$ as a \textit{W-surface}, whose precise definition is as follows (see \cite[Section 2]{U16a}).

\begin{definition}
A \textit{W-surface} is a normal projective surface $X$ together with a proper deformation $(X \subset \X) \to (0 \in \D)$ such that
\begin{enumerate}
\item $X$ has at most Wahl singularities.
\item $\X$ is a normal complex $3$-fold with $K_{\X}$ $\Q$-Cartier.
\item The fiber $X_0$ is reduced and isomorphic to $X$.
\item The fiber $X_t$ is nonsingular for $t\neq 0$.
\end{enumerate}
The W-surface is said to be \textit{smooth} if $X$ is nonsingular.
\label{wsurf}
\end{definition}

\begin{remark}
One could have also considered more general smoothings $(X \subset \X) \to (0 \in \D)$ where $X$ has just quotient singularities. But we can always replace that situation with a W-surface via base change, simultaneous resolutions, P-resolutions, and M-resolutions (see \cite[Lemma 5.2]{HTU13}). Hence $\Q$-Gorenstein smoothings of surfaces with only Wahl singularities are the key deformations in this type of degenerations.
\end{remark}

For a W-surface the invariants $q(X_t)$, $p_g(X_t)$, $K_{X_t}^2$, $\chi_{\text{top}}(X_t)$ remain constant for every $t \in \D$. A W-surface is \textit{minimal} if $K_X$ is nef, and so $K_{X_t}$ is nef for all $t$. If a W-surface is not minimal, then we can run explicitly the MMP relative to $\D$ \cite{HTU13}, arriving to a minimal model or other outcomes as explained in \cite[Sect. 2]{U16a}. When $K_X$ is nef and big, the canonical model of $(X \subset \X) \to (0 \in \D)$ has only T-singularities (RDP and $\frac{1}{dn^2}(1,dna-1)$ are both possible of course). For details we refer to \cite[Sect. 2]{U16a} and \cite[Sect. 2 and 3]{U16b}.

We now explore closely the situation $p_g=1$, $q=0$.

\begin{proposition}
Let $X$ be a minimal singular W-surface (indeed singular) with $p_g(X)=1$ and $q(X)=0$. Let $\phi \colon \tilde{X} \to X$ be its minimal resolution, and let $\pi \colon \tilde{X} \to S$ be a composition of blow-ups so that $S$ has no $(-1)$-curves. Then $p_g(S)=1$, $q(S)=0$, $\pi(\exc(\phi))$ is a nonempty configuration of rational curves, and $S$ is one of the following:
\begin{enumerate}
\item A K3 surface.
\item A Kodaira dimension $1$ elliptic surface $S \to \P^1$ with at least one multiple fiber. If moreover $K_X$ is ample, then $\exc(\phi)$ contains some multiple section of $S \to \P^1$.
\item A surface of general type with $K_S^2 < K_X^2$.
\end{enumerate}
\label{q=0pg=1}
\end{proposition}

\begin{proof}
In \cite[Prop.2.3]{RU19} we have a general analogue statement for $X$ with $1$ Wahl singularity, here we may have many. As Wahl singularities are rational, we have that $p_g$ is preserved, and so $S$ cannot be a ruled surface, $p_g(S)=1$, and $q(S)=0$. By the Enriques' classification, we know that $S$ is either a K3 surface, a Kodaira dimension $1$ surface, or a surface of general type. By \cite[Lemma 2.4]{K92}, if $S$ is of general type, the $K_S^2 < K_X^2$.

We can write $$\phi^*(K_X) \equiv \pi^*(K_S) + \sum_i \beta_i E_i + \sum_j \alpha_j C_j$$ where $E_i$ are exceptional curves of $\pi$, $C_j$ are exceptional curves of $\phi$, $\beta_i$ are positive integers, and $\alpha_j$ are minus the discrepancies of $C_j$, and so positive rational numbers in $(0,1)$. If $\pi(\exc(\phi))$ had no curves, then the $C_j$ would all be contained in the $E_i$. By Mumford's criterion (cf. \cite[(2.1)Theorem, III]{BHPV04}), we would have $0 \leq K_X^2 = K_S^2 - N$ where $N$ is a positive rational number. Hence we would obtain a contradiction if $\kappa(S) \leq 1$, since $K_S^2=0$. For $S$ of general type, we would obtain  $0 \leq K_X^2 = K_S^2 - N< K_X^2 -N$ which is a contradiction. Therefore $\pi(\exc(\phi))$ is a nonempty configuration of rational curves in $S$.

Finally let $\kappa(S)=1$. By the canonical class formula (cf. \cite[(12.3)Corollary, V]{BHPV04}), we can say that $K_S \sim \sum_i (m_i-1)F_i$ where $m_i F_i$ are the multiple fibers of $S \to \P^1$ (with multiplicities $m_i>1$). As $K_S$ is not trivial, there must be multiple fibers. If $K_X$ is in addition ample, then we must have a multiple section in $\pi(\exc(\phi))$, since otherwise we obtain a zero curve for $K_X$ by pulling back a general fiber.
\end{proof}

In \cite[Sect. 4]{K92}, Kawamata characterizes the central singular fiber of a minimal W-surface with general fiber of Kodaira dimension $0$ or $1$. As an example, we can start with a K3 surface $S$ with an elliptic fibration $S \to \P^1$ which has a $I_1$ fiber $F$. As in \cite[Sect. 4]{U16b}, we can construct any given Wahl chain from successive blow-ups over the node of $F$, and then contract it to a surface $X$. One can easily prove that $X$ has no-local-to-global obstructions to deform (see below), $K_X$ is nef, and the general fiber of a $\Q$-Gorenstein smoothing is a Kodaira dimension $1$ surface with an elliptic fibration with one smooth multiple fiber, whose multiplicity is the index of the singularity (see \cite[Ex. 4.5]{K92}). In comparison, the situation when the smooth fiber is of general type is probably impossible to classify.

One can construct examples for each of the cases in Proposition \ref{q=0pg=1}. For instance, we will see many from K3 surfaces in this paper. For Kodaira dimensions $1$ and $2$, the situation is more difficult but there are plenty of examples. We will show some via MMP starting with the situation of a K3 W-surface, using the identification strategy in \cite{U16b}.

\begin{remark}
If moreover the general fiber $X_t$ in Proposition \ref{q=0pg=1} part (2) is simply-connected, then the surface $S$ is simply-connected (see e.g. \cite[proof of Prop.6.1]{RTU17}). In this way, the elliptic fibration $S \to \P^1$ may have at most two multiple fibers with coprime multiplicities and it has a singular fiber (see e.g. \cite[Theorem 4]{X91}).
\label{pi1=1}
\end{remark}

If we fix the $K_X^2$ for W-surfaces $X$ with ample canonical class, then we have finitely many possibilities for Wahl singularities in $X$ by Alexeev's boundedness. The concrete finite list is unknown in general, even if we fix more invariants. However there are fairly good bounds on the length of the Wahl chains (and so on the index of the singularity) for approaching this problem. The bounds are better for low $K^2$, and so we will attempt classification in the case of $p_g=1$, $q=0$, and $K^2=1$.

In the next proposition, we recall from \cite{RU19} a broad classification for $p_g=1$ W-surfaces with one Wahl singularity and ample canonical class. In complete generality, it can be proved that for surfaces with more than one Wahl singularities we have the general bound $$\ell \leq 4 K^2+1.$$ This is an improvement of the bound $\ell \leq 4 K^2+7$ found by Evans and Smith in \cite{ES20}, which can be achieved through minor modifications of two propositions in \cite{ES20} using ideas from \cite{RU19}. This will be done somewhere else.

\begin{proposition}
Let $X$ be a W-surface with one Wahl singularity, $p_g=1$, $q=0$, and $K_X$ ample. Let $S$ be the minimal surface defined in Proposition \ref{q=0pg=1}. Then,
\begin{enumerate}
\item If $S$ is a K3 surface, then $\ell \leq 4K_X^2+1$.
\item If $S$ is properly elliptic surface, then $\ell \leq 4K_X^2-1$.
\item If $S$ is of general type, then $$\ell \leq 4(K_X^2-K_S^2)-3$$ when $K_X^2-K_S^2>1$, otherwise $\ell \leq 2$.
\end{enumerate}
\label{OneWahlq=0pg=1}
\end{proposition}

\begin{proof}
This is the main theorem in \cite{RU19} and Proposition \ref{q=0pg=1}.
\end{proof}

We point out that the proof of Proposition \ref{OneWahlq=0pg=1} does not require any deformations. Better bounds can be obtained when the configuration of exceptional curves has no ``long diagrams" (see \cite[Introduction]{RU19}), and \cite{RU19} contains a classification in the case when equality holds.

We now consider the smallest case $K^2=1$. As pointed out in the introduction, Catanese proves in \cite{C80} that the moduli space of $K^2=p_g=1$ surfaces of general type is an irreducible rational variety of dimension $18$. These surfaces are simply-connected, and they can be seen as weighted complete intersections $(6,6)$ in $\P(1,2,2,3,3)$. In \cite{FPR17} it is proved that the moduli space of Gorenstein such surfaces (which also contain the smooth ones) is an irreducible rational variety of dimension $18$. In \cite[Sect. 5]{FPR17} they show several examples, normal and non-normal (see also \cite{DoRo21}). They also have some non-Gorenstein normal examples with Wahl singularities of type $\frac{1}{4}(1,1)$. The complete closure of the moduli space of $K^2=p_g=1$ surfaces of general type in the KSBA moduli space seems to be unknown. Unobstructed surfaces with one Wahl singularity would produce divisors in this KSBA compactification. The following theorem classifies those surfaces, to then shortly realize some of them.

\begin{theorem}
Let $X$ be a W-surface with $K_X^2=1$, $p_g=1$ and $q=0$, which has one Wahl singularity and $K_X$ ample. If $\phi \colon \tilde{X} \to X$ is the minimal resolution (with exceptional divisor $\exc(\phi)$), then $X$ belongs to the following list:

\begin{itemize}
\item[A.] $\kappa(\tilde{X})=1$:

\begin{enumerate}
\item[$(A1)$] $\tilde{X}$ is a $D_{3}$, $\exc(\phi)=[4]$, and the $(-4)$-curve is a $3$-section.

\item[$(A2)$] $\tilde{X}$ is a $D_{2}$, $\exc(\phi)=[4]$, and the $(-4)$-curve is a $4$-section.

\item[$(A3)$] $\tilde{X}$ is the blow-up at one point of a $D_{2}$, and $\exc(\phi)=[5,2]$. The $(-1)$-curve intersects the $(-5)$-curve with multiplicity $2$ (disjoint to the $(-2)$-curve).

\item[$(A4)$] The surface $\tilde{X}$ is the blow-up of a $D_{2}$ twice at the node of a multiplicity $2$ $I_1$ fiber, and $\exc(\phi)=[3,5,2]$. The surface $D_{2}$ contains a $(-3)$-curve which is a $2$-section.

\end{enumerate}

\item[B.] $\kappa(\tilde{X})=0$, and so $\tilde{X}$ is a K3 surface blown-up:

\begin{enumerate}
\item[$(B1)$] once, and $\exc(\phi)=[5,2]$. The $(-1)$-curve intersects the $(-5)$-curve with multiplicity $3$.

\item[$(B2)$] twice, and $\exc(\phi)=[2,5,3]$. There is one $(-1)$-curve touching the $(-5)$-curve with multiplicity $2$, and there is another $(-1)$-curve intersecting the $(-5)$-curve and the $(-3)$-curve at one point.

\item[$(B3)$] twice, and $\exc(\phi)=[6,2,2]$. There are two disjoint $(-1)$-curves intersecting the $(-6)$-curve with multiplicity $2$ each.

\item[$(B4)$] three times, and $\exc(\phi)=[2,6,2,3]$. There is a $(-1)$-curve intersecting the first $(-2)$-curve and the $(-6)$-curve at one point each, and there is a $(-1)$-curve intersecting the $(-6)$-curve and the $(-3)$-curve at one point each.

\item[$(B5)$] three times, and $\exc(\phi)=[3,5,3,2]$. There is a $(-1)$-curve intersecting the first $(-3)$-curve and the $(-5)$-curve at one point each, and there is a $(-1)$-curve intersecting the $(-5)$-curve and the $(-2)$-curve at one point each.

\item[$(B6)$] four times, and $\exc(\phi)=[2,2,6,2,4]$. There is a $(-1)$-curve intersecting the first $(-2)$-curve and the $(-6)$-curve at one point each, and there is a $(-1)$-curve intersecting the $(-4)$-curve with multiplicity $2$.

\end{enumerate}
\end{itemize}

\label{OneWahlK=1}
\end{theorem}

\begin{proof}
By Proposition \ref{q=0pg=1}, the only possibilities are $\kappa(\tilde{X})=1$ or $0$, since $K_X^2=1$.  By \cite{C80} and Remark \ref{pi1=1}, we know that $\tilde{X}$ is simply-connected. 

In the case $\kappa(\tilde{X})=1$, we have that the minimal model $S$ has an elliptic fibration $S \to \P^1$ with at most two coprime multiple fibers. Let $a$ and $b$ be the multiplicities, and so $K_S \sim (a-1)F_a + (b-1) F_b$, where $F_a$ and $F_b$ are the reduced multiple fibers. By Proposition \ref{OneWahlq=0pg=1}, we know that the length $\ell$ is at most $3$. The case $\ell=3$ is worked out in \cite[Theorem 3.2]{RU19}, which only gives (A4). For $[4]$ we have $\tilde{X}=S$, and for $[5,2]$ we have $\tilde{X}=Bl_{pt}(S)$, since $K_S^2=0$. In the case $[4]$, there must be a $(-4)$-curve $\Gamma$ which is a $uab$-section in $S$ that satisfies $$2 = \Gamma \cdot K_S = u(2ab-a-b).$$ This has no solutions unless $a$ or $b$ is $1$. Say $b=1$. Then either $a=3$, which means $\Gamma$ is a $3$-section and we are in case (A1), or $a=2$, where $\Gamma$ is a $4$-section and we are in case (A2). A similar analysis gives (A3) for the option $[5,2]$.

When $\kappa(\tilde{X})=0$, we have that $\ell \leq 5$, and we can do precisely the analysis in \cite[Theorem 2.1]{DRU20}, replacing the Enriques surface by a K3 surface. This way, we arrive at seven possible situations, which are pictured in \cite[Figure 2]{DRU20}. But one of them is obviously impossible for a K3 surface, where there is an $I_2$ configuration intersecting transversally at one point a $I_1$ curve. The $I_2$ would be a singular fiber for some elliptic fibration and the $I_1$ would be a singular section, which is not possible. All in all, we are left with (B1) -- (B6) in our claim.
\end{proof}

\begin{remark}
When we have a W-surface with many Wahl singularities and no-local-to-global obstructions to deform, then in particular there are W-surfaces with one singularity, and we are in Theorem \ref{OneWahlK=1}. This is done via partial $\Q$-Gorenstein smoothings. But if we do have obstructions, then those singularities may not even appear in the list of Theorem \ref{OneWahlK=1}. On the other hand, in terms of only lengths, we know that in general $\ell \leq 5$ for this case, and so we may have all Wahl chains up to length $5$ via obstructed surfaces with two or more Wahl singularities. At the end of this section we will illustrate what we know.
\label{boo}
\end{remark}

Classification of W-surfaces for higher $K^2$ is much harder. At the same time, much less is known about the corresponding moduli spaces. For example, the moduli space of $p_g=1$, $q=0$ surfaces with $K^2=2$ is not known completely. As it was said in the introduction, it can be showed that the torsion of the Picard group is either $\{0\}$ or $\Z/2$. For the latter, Catanese and Debarre \cite{CD80} proved that the moduli space of such surfaces is a rational irreducible variety of dimension $16$, and these surfaces have $\pi_1=\Z/2$. They also worked out various properties for torsion zero surfaces, including an irreducible $16$ dimensional component of the moduli space, but it is still open if this component is indeed the moduli of all torsion zero surfaces. For $p_g=1$, $q=0$ surfaces with $K^2=3$ much less is known, the torsion of the Picard group has order at most $4$, and the ones with torsion $\Z/3$ are described in \cite{M03}. All what is known for simply-connected $p_g=1$ surfaces of general type with $K^2 \geq 3$ are the examples in \cite{PPS13}.

We now show how to construct W-surfaces from nodal configurations (not necessarily SNC but only nodes as singularities) of rational curves in K3 surfaces. After that, we end this section applying Proposition \ref{obstruction} to construct most of the situations in Theorem \ref{OneWahlK=1}.

\begin{proposition}
Let $S$ be a K3 surface, and let $\AR$ be a nodal configuration of $r$ rational curves in $S$. Assume there is a composition of blow-ups $\pi \colon \tilde X \to S$ over some nodes and infinitely near nodes of $\AR$ (considering also nodes in irreducible curves), so that $\pi^*(\AR)$ contains disjoint Wahl Chains which are contracted to a surface $X$. Then the dimension of $H^2(X,T_X)$ is $r$ minus the dimension of the $\Q$-vector space generated by $\AR$ in the N\'eron-Severi group of $S$. In particular, if $\AR$ is linearly independent, then there are no-local-to-global obstructions to deform $X$.
\label{obstruction}
\end{proposition}

\begin{proof}
Let $\AR =\{C_1, \ldots,C_r \}$. Let $\pi' \colon S' \to S$ be the blow-up of all nodes in the irreducible curves $C_i$, and define $\AR'=\{C'_1, \ldots,C'_r \}$ via their proper transforms. Let $E_1,\ldots, E_s$ be the exceptional $(-1)$-curves of $\pi'$, and let $\AR''=\AR' + \sum E_i$. By \cite[2.3(c)]{EV92}, we have the exact sequence $$0 \to \Omega_{S'}^1(\log \AR'')  \to  \Omega_{S'}^1 (\log \AR') \otimes \Big(\sum E_i\Big) \to \bigoplus_{E_i} \Omega_{E_i}^1 \otimes \AR''|_{E_i} \to 0.$$ Since $K_{S'} \sim \sum E_i$, by Serre's duality we have $$H^0(S',\Omega_{S'}^1(\log( \AR'')))=H^2(S',T_{S'}(-\log \AR'')),$$ because the degree of $\Omega_{E_i}^1 \otimes \AR''|{E_i}$ is $-1$ (here we are using that we resolve only nodes). We now consider the residue sequence \cite[2.3(a)]{EV92} $$0 \to \Omega_{S'}^1  \to  \Omega_{S'}^1(\log \AR'') \to \bigoplus_{\Gamma \in \AR''} \O_{\Gamma} \to 0.$$ As $H^0(S',\Omega_{S'}^1)=0$, we have the exact sequence of morphisms in cohomology $$ 0 \to H^0(S',\Omega_{S'}^1(\log \AR'')) \to \bigoplus_{\Gamma \in \AR''} H^0(\Gamma,\O_{\Gamma}) \to H^1(S',\Omega_{S'}^1),$$ where the last morphism is the first Chern class map, giving as images the classes of $\Gamma$ in the N\'eron-Severi group of $S'$. We now note that the dimension of $H^0(S',\Omega_{S'}^1(\log \AR''))$ is $r$ minus the dimension of the $\Q$-vector space generated by $\AR$ in the N\'eron-Severi group of $S$, as the exceptional curves $E_i$ are independent.

Let $\pi'' \colon \tilde{X} \to S'$ be the rest of the blow-ups to reach $\pi \colon \tilde{X} \to S$. We start with $H^2(S',T_{S'}(-\log \AR''))$. From here, we maintain the dimension if we  either erase at any step a $(-1)$-curve transversal to the rest of the configuration (e.g. any $E_i$), or we add $(-1)$-curves from blow-ups at the nodes of the configuration (cf. \cite{LP07}). Therefore, if $\AR_j$ are the Wahl chains in $\pi^*(\AR)$ to be contracted to $X$, then $H^2(S',T_{S'}(-\log \AR'')) \simeq H^2(\tilde X,T_{\tilde X}(-\log \sum \AR_j))$. On the other hand, by \cite[Theorem 2]{LP07}, we have $H^2(X,T_X) \simeq H^2(\tilde X,T_{\tilde X}(-\log \sum \AR_j))$, and so we have what is claimed. If $\AR$ is linearly independent, then $H^2(X,T_X)=0$ and every local deformation of its singularities may be globalized \cite[Prop. 4]{LP07}.
\end{proof}

\begin{remark}
One can easily obtain an analogue criteria to Proposition \ref{obstruction} for Enriques surfaces by using the following fact (see \cite[Section 3.1]{RU19}). Let $S_0$ be an Enriques surface, and let $f \colon S \to S_0$ be the \'etale double cover induced by $2K_{S_0} \sim 0$. Let $\AR_0$ be a SNC divisor on $S_0$, and $\AR=f^*(\AR_0)$. Then $$H^0(S,\Omega_{S}^1(\log \AR )) \simeq H^2(S_0,T_{S_0}(-\log \AR_0 )) \oplus H^0(S_0,\Omega_{S_0}^1(\log \AR_0 )).$$ Hence $\AR$ linearly independent implies no obstruction for any construction from $(S_0,\AR_0)$.
\end{remark}

We now show what we can realize from Theorem \ref{OneWahlK=1}. The cases (B2), (B3), (B4), (B5) correspond to nodal configurations of rational curves which can be realized, and moreover, it can be easily proved that they are linearly independent in the Ner\'on-Severi group. The case (B6) consist of two nodal $I_1$ fibers and one section. This can be realized of course, but it has a one dimensional obstruction space. We can construct (B1) just as we did for (B1) in \cite[Sect. 2.3]{DRU20}. We note that in the construction of (B1) in \cite{DRU20} we used flips, and so it is not direct, but it proves that the corresponding W-surface has no obstructions. Thus it is not necessary to start with nodal divisors to obtain unobstructed surfaces.

\begin{figure}[htbp]
\centering
\includegraphics[width=10cm]{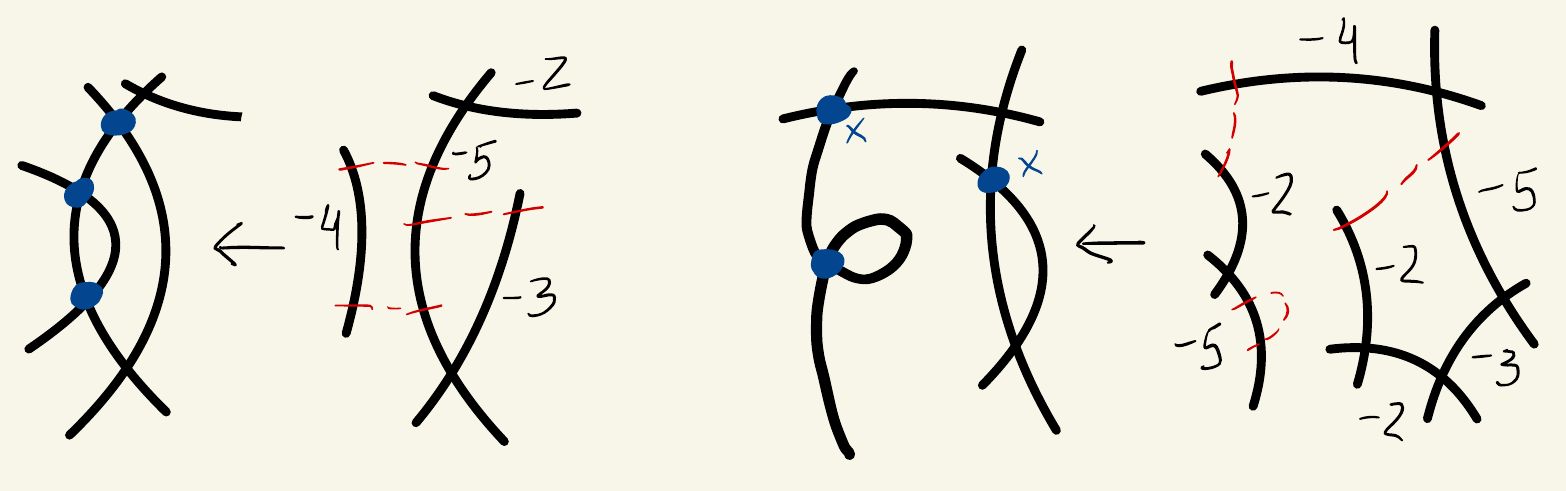}
\caption{Configurations for (A2) (left) and extra Wahl singularity (right)}
\label{f0}
\end{figure}

For the (An) cases, we are able to realize only (A2). We start with a configuration of four $(-2)$-curves as in Figure \ref{f0} (left), and then we blow-up as shown. The contraction of both Wahl chains has $K$ ample, $K^2=1$. This situation has no-local-to-global obstructions, and so it produces a surface with $K$ ample and one $\frac{1}{4}(1,1)$ singularity. Using flips as in \cite{U16b}, we can show that this is situation (A2). We do not know about existence for the cases (A1), (A3), and (A4). We finish with the example shown in Figure \ref{f0} (right), which produces a surface with Wahl singularities corresponding to $[2,5]$ and $[2,2,3,5,4]$, and $K$ is ample with $K^2=1$. It is of course obstructed, since it is using two fibers $I_1$ and $I_2$ of an elliptic fibration (and a $(-2)$-curve which is a section). Another way to see that it is obstructed: the Wahl chain $[2,2,3,5,4]$ is not part of Theorem \ref{OneWahlK=1} (see Remark \ref{boo}). 

\section{Rational configurations in K3 surfaces} \label{s3}

In the construction of W-surfaces with ample canonical class, i.e. surfaces $X$ with only Wahl singularities and $K_X$ ample that have a $\Q$-Gorenstein smoothing, we typically start with a special configuration of rational curves $\AR$ in some minimal surface $S$. We then consider particular blow-ups $\pi \colon \tilde X \to S$, so that the pull-back of $\AR$ contains Wahl chains which are contracted to obtain a singular surface $X$. Let $\phi \colon \tilde X \to X$ be the contraction of the chains. When we see this backwards, what we really are thinking is that $\AR$ is the image of $\exc(\phi)$ by $\pi$. This is a delicate operation, since in addition we want $K_X$ ample, and the existence of $\Q$-Gorenstein smoothings. \textit{Which pairs $(S,\AR)$ exactly occur after fixing invariants?} This is the geographical question that motivates this section in the case of $p_g=1$, $q=0$ surfaces of general type, starting with a K3 surface $S$ and a simple normal crossings configuration of $(-2)$-curves $\AR$. This particular situation is enough for our purposes. In a sequel work we will consider more general arrangements.

Let $S$ be a K3 surface which contains a simple normal crossings configuration $\AR$ of $r$ $(-2)$-curves. Let $t_2$ be the number of nodes in $\AR$. Let $\Omega_S^1(\log \AR)$ be the logarithmic differentials on $S$ with simple poles at $\AR$. It is a rank $2$ locally free sheaf, whose Chern classes define the log Chern numbers $$\bar{c}_1^2:= \AR^2= 2t_2-2r \ \ \ \ \ \ \bar{c}_2:= \chi_{top}(S)- \chi_{top}(\AR)=24+t_2-2r.$$

\begin{proposition}
Let us assume that the process above works, i.e from a configuration $\AR$ of $r$ $(-2)$-curves we construct an $X$ with $K_X$ ample and $P$ Wahl singularities. Then
$$\bar{c}_1^2= 2 K_X^2 \ \ \ \ \ \ \bar{c}_2= 24 -P-K_X^2,$$ and $$K_X^2 \leq 14 - \frac{1}{5}(3P-2).$$ In this way, we have $r=P+2K_X^2$ and $t_2=3K_X^2+P$. The number of nodes in $\AR$ to be blown-up is $P+K_X^2$.
\label{chern}
\end{proposition}

\begin{proof}
Let us first consider an arbitrary SNC configuration of smooth rational curves $\AR=\{C_1,\ldots,C_r \}$ in a surface $S$ with $t_2$ nodes as singularities. Let us define $$P(\AR):= \sum_{i=1}^r C_i^2 +5r-2t_2 \ \ \ \text{and} \ \ \ K(\AR):= K_S^2 +2r-t_2-P(\AR).$$ Let $\sigma \colon S' \to S$ be the blow-up at a node of $\AR$, let $E$ be the exceptional $(-1)$-curve, and let $\AR' \subset S'$ be the proper transform of $\AR$. Then an easy computation shows that $P(\AR'+E)=P(\AR)$ and $K(\AR'+E)=K(\AR)$. At the same time, if $E$ is a $(-1)$-curve in $\AR$ which intersects $\AR - E$ at only two points, then $P(\AR - E)=P(\AR)$ and $K(\AR - E)=K(\AR)$. 

On the other hand, if $\{A_1,\ldots,A_s\}$ is a Wahl chain, then $$\sum_{i=1}^s A_i^2 =-3s-1.$$ Hence if $\AR$ is the disjoint union of $P$ Wahl chains of lengths $\ell_i$ in a surface $\tilde{X}$ and $X$ is the contraction of the chains, then $P(\AR)= P$ and $K(\AR)= K_{\tilde{X}}^2 + \sum_{i=1}^P \ell_i=K_X^2$.

Now let $\AR \subset S$ be the SNC configuration of $(-2)$-curves in the K3 surface $S$. In the process of constructing the Wahl chains in $\tilde{X}$ through the composition of blow-ups $\pi \colon \tilde{X} \to S$ we can only blow-up at nodes and subsequent nodes of proper transforms of $\AR$, since otherwise we would violate the ampleness of $K_X$. This is because, by blowing up away from the nodes we would obtain a $(-1)$-curve that intersects transversally the Wahl chains in at most one point, and so the image of that curve in $X$ would be negative for $K_X$. As $P(\AR)$ and $K(\AR)$ remain invariants by adding or subtracting exceptional $(-1)$-curves, we obtain that
$P(\AR)= P$ and $K(\AR)=K_X^2$. 

But we know that $\bar{c}_1^2= 2t_2-2r$, and $\bar{c}_2=24+t_2-2r$, and so we obtain $\bar{c}_1^2= 2 K_X^2$ and $\bar{c}_2= 24 -P-K_X^2$. 

To show the inequality $K_X^2 \leq 14 - \frac{1}{5}(3P-2)$ we directly apply \cite[Theorem 0.1]{L03} to the log canonical pair $(S,\AR)$, since $K_S + \AR = \AR$ is effective.
\end{proof}

We briefly remark that typically the divisor $K_S+\AR=\AR$ is nef, big and ample modulo curves in $\AR$. So, in that case we can say by \cite[Thm. 3.1]{TY87} that $K_X^2 = 14 - \frac{1}{5}(3P-2)$ if and only if $S \setminus \AR$ is a ball quotient, and this is not possible as $\AR$ would need to be a disjoint union of exceptional divisors of quotients of elliptic singularities. 

The reader may be wondering why do we not just consider the situation $(\tilde{X},\exc(\phi))$ to compute the log Chern numbers and the corresponding inequality. The reason is that we obtain weaker invariants. In that case, $\bar c_1^2= K_X^2 - P$, $\bar c_2= 24-K_X^2-P$, and so $K_X^2 \leq 18 - \frac{1}{2} P$ by the Langer-BMY inequality \cite[Theorem 0.1]{L03}.

\begin{remark}
We point out that the analogue situation to Proposition \ref{chern} works on an Enriques surface. In that case $\bar{c}_1^2= 2 K_X^2$, $\bar{c}_2= 12 -P-K_X^2$, and so $$K_X^2 \leq 7 - \frac{1}{5}(3P-1).$$ Therefore, with one or two Wahl singularities we may get $K_X^2=6$, but for more the maximum possible $K_X^2$ is $5$. In the case of $P=2$, we can disregard $K_X^2=6$ through \cite[Thm. 3.1]{TY87} if $\AR$ is nef, big and ample modulo curves in $\AR$.
\end{remark}

\begin{proposition}
Let $S$ be a K3 surface, and let $\AR$ be a SNC configuration of $(-2)$-curves. Assume there is a composition of blow-ups $\pi \colon \tilde X \to S$ so that the pull-back of $\AR$ contains disjoint Wahl Chains. Let $\phi \colon \tilde X \to X$ be the contraction of them. Let $\exc(\phi)=\sum_{i} C_i$ the decomposition as sum of $\P^1$ (i.e. the sum of the curves in the Wahl chains), and let $d_i$ be the discrepancy of $C_i$. Assume $\AR=\pi(\exc(\phi))$. Then, we have

\begin{itemize}
\item[(n)] $K_X$ is nef if and only if for any $(-1)$-curve $\Gamma$ in $\tilde{X}$ -- thus contained in $\pi^*(\AR)$ -- we have $-1 \geq d_i+d_j$, where $\Gamma \cdot C_i=\Gamma \cdot C_j=1$.

\item[(a)] $K_X$ is ample if and only if $K_X^2>0$, for any $(-1)$-curve in $\tilde{X}$ we have $-1 > d_i+d_j$, where $\Gamma \cdot C_i=\Gamma \cdot C_j=1$, and there are no $(-2)$-curves disjoint from $\pi^*(\AR)$.
\end{itemize}

Moreover, if we satisfy case \textnormal{(n)}, $K_X^2>0$, and $X$ has $\Q$-Gorenstein smoothings, then we can find the canonical model of $X$ by contracting all $(-1)$-curves in $\tilde{X}$ with $-1 = d_i+d_j$, and all $(-2)$-curves disjoint from $\pi^*(\AR)$. The resulting surface has only T-singularities.
\label{nef}
\end{proposition}

\begin{proof}
Let $K_{\tilde X} \sim \sum_i \alpha_i E_i$ where $E_i$ are the proper transforms of each of the blow-ups from $\pi$, and $\alpha_i \in \Z_{>0}$. We know that numerically $K_{\tilde X} \equiv \phi^*(K_X) + \sum_i d_i C_i$. We recall that $-1<d_i<0$. Therefore $\phi^*(K_X)$ is a $\Q$-effective divisor, supported on the $C_i$ and the $E_j$ (which of course may coincide). Hence a necessary and sufficient condition for $K_X$ to be nef is that $\phi^*(K_X) \cdot E_i \geq 0$ for all $i$.

We claim first that $\pi$ can only blow-up at nodes of $\AR$ and infinitely near nodes, since otherwise there would be a $(-1)$-curve intersecting $\exc(\phi)$ transversally in at most one point, so $K_X$ would not be nef. For the same reason, the only curves $E_i$ not equal to some $C_j$ are the $(-1)$-curves. Therefore, we only need to verify $\phi^*(K_X) \cdot E_i \geq 0$ at the $E_i$ which are $(-1)$-curves, and this is $-1 \geq d_i+d_j$. We do not have more than two intersections because $\AR$ is a SNC divisor.

For ampleness we use the Nakai-Moishezon criterion, which requires $-1 > d_i+d_j$ for all $(-1)$-curves, and any zero curve for $\phi^*(K_X)$ (not a $C_j$) to be disjoint from $\sum_i E_i + \sum_j C_j$. Let $X \to X_{can}$ be the canonical model of $X$ given by $|mK_{X}|$ for $m \gg 0$. Then we have that any zero curve is a $(-2)$-curve, as they do not pass through the singularities of $X$. 

In case of (n), $K_X^2>0$, and $\Q$-Gorenstein smoothing, we have that the canonical model of the $\Q$-Gorenstein smoothing has as special fiber the canonical model of $X$ with only T-singularities, as shown in \cite[Prop. 2.7]{U16a}.
\end{proof}

\begin{corollary}
Let us assume the same hypothesis in Proposition \ref{nef}, $K_X$ nef, and that the intersection matrix of $\AR$ is invertible. Then $X$ has no-local-to-global obstructions, and so there are $\Q$-Gorenstein smoothings for $X$. Moreover, if $K_X^2>0$, then the canonical model $X_{can}$ is a point in the KSBA compactification of the moduli space of surfaces of general type with $p_g=1$, $q=0$, and $K^2=K_X^2$, which belongs to a $(20-2K_X^2)$-dimensional irreducible family. Each Wahl singularity in $X_{can}$ produces an irreducible divisor which parametrizes KSBA stable surfaces with that Wahl singularity.
\label{KSBA}
\end{corollary}

\begin{proof}
We can apply Proposition \ref{obstruction} as $K_X$ is nef (so we blow-up only at nodes and infinitely near nodes of the SNC divisor $\AR$), and the intersection matrix of $\AR$ is invertible. So $X$ has no-local-to-global obstructions to deform. If $K_X^2>0$, then, according to Proposition \ref{nef}, to obtain the canonical model $X_{can}$ we need to contract a bunch of $(-1)$-curves attached to two Wahl chains, and contract ADE configurations of $(-2)$-curves disjoint from $\AR$. The contraction of the $(-1)$-curves would produce non-ADE T-singularities, and the contraction of ADE chains produce RDP singularities. So instead we can start with $\AR$ union the ADE disjoint configurations, which would also give an invertible intersection matrix, and blow-up to obtain the T-chains, to use the same proof of Proposition \ref{obstruction} to prove that $X_{can}$ has no-local-to-global obstructions to deform. The invariants of the smoothings are clear. For the final statement we refer to \cite{H12}.
\end{proof}

In particular, to find surfaces $X$ as in Corollary \ref{KSBA} for a fixed $K^2$, we have the restriction $1 \leq K_X^2 \leq 9$. In addition, the K3 surface must have Picard number at least $r=P+2 K_X^2 \geq 3$, where $P$ is the number of Wahl singularities in $X$. Thus for higher $K^2$ there should be few Wahl singularities, and we would need a K3 surface with high Picard number and a configuration $\AR$ with high log Chern slope. That poses a difficulty to find $X$ with high $K^2$, since high log Chern slopes are harder to achieve, and high Picard number decreases the dimension of the loci of the corresponding K3 surfaces. 

\begin{remark}
In fact, say that the surface $X$ in Corollary \ref{KSBA} has $K_X$ ample, and only one Wahl singularity. The configuration $\AR$ of $r$ $(-2)$-curves defines a sub-lattice $M$ in Pic$(S)$, which is even, nondegenerate, and of signature $(1,r-1)$. Then, by \cite[Proposition 2.1]{D96}, it defines a local moduli space $S_M$ of isomorphism classes of M-polarized K3 surfaces of dimension $20-r$. By Proposition \ref{chern} we have that $r=1+2K_X^2$, and so the dimension of $S_M$ is $20-(1+2K_X^2)=19-2K_X^2$, which is the dimension of the divisor corresponding to $X$ in the KSBA compactification of the moduli of surfaces of general type with $p_g=1$, $q=0$, and $K^2=K_X^2$. There are many such divisors.
\label{igor}
\end{remark}

We finish this section with the recipe we will use to compute the fundamental group of a nonsingular fiber of a W-surface $X$.

\begin{proposition}
Let $X$ be a W-surface. Let $\phi \colon \tilde X \to X$ be the minimal resolution of $X$, and so $\exc(\phi)$ is a disjoint union of Wahl chains $\mathcal{C}_{n,a}$, where the associated Wahl singularity is $\frac{1}{n^2}(1,na-1)$. Assume that for each $\mathcal{C}_{n,a}$ there is either 
\begin{itemize}
\item[(I)] a $\mathcal{C}'_{n',a'}$ with $\gcd(n,n')=1$ and a $C \simeq \P^1$ such that $C$ intersects only $\mathcal{C}_{n,a}$ and $\mathcal{C}'_{n',a'}$ transversally at one point in one of their ends, or 

\item[(II)] a $C \simeq \P^1$ such that $C$ intersects only $\mathcal{C}_{n,a}$ transversally at one point in one of its ends.
\end{itemize}

Then $\pi_1(\tilde X) \simeq \pi_1(X_t)$ for all $t \in \D$. In this way, after running MMP on the $3$-fold family relative to $\D$, and after taking the canonical model (if possible), the fundamental group of the general fiber of the resulting $\Q$-Gorenstein deformation is isomorphic to $\pi_1(\tilde X)$.
\label{pi1}
\end{proposition}

\begin{proof}
The computation of $\pi_1(X_t)$ can be done as in \cite[Theorem 3]{LP07}, it uses the surgery involved in a $\Q$-Gorenstein smoothing and Van-Kampen Theorem. The last statement follows from the fact that T-singularities are rational, and on the general fiber in the MMP process or the canonical model (if possible) we have either isomorphisms, blow-down of $(-1)$-curves, or contraction of ADE configurations, all of which preserve the isomorphism class of $\pi_1$.
\end{proof}

\section{Simply-connected $p_g=1$ surfaces for $2 \leq K^2 \leq 9$} \label{s4}

By definition, a K3 surface with an elliptic fibration with sections is called \textit{extremal} if it has Picard number $20$ and its Mordell-Weil group is finite. (A classification of all extremal K3 surfaces can be found in \cite{Ye99,SZ01}, together with their Mordell-Weil groups.) An extremal K3 surface with only singular fibers of Kodaira type $I_n$ (i.e. semi-stable) must have $6$ singular fibers. In this section we will consider configurations $\AR$ inside of a fixed extremal K3 surface $S$. The $\AR$ will be subconfigurations of the configuration $\AR_0$, which consists of all singular fibers and all sections.

\begin{remark}
Say that all fibers are of Kodaira type $I_n$, and let $MW$ be the order of the Mordell-Weil group. Then we have $$r(\AR_0)=24 + MW \ \ \ \text{and} \ \ \ t_2(\AR_0)=24 +6MW,$$ since torsion sections must be disjoint. Thus $\bar{c}_1^2(\AR_0) =10 MW$, $\bar{c}_2(\AR_0)=4 MW$, and so $$\bar{c}^2_1(\AR_0)/\bar{c}_2(\AR_0)=\frac{5}{2}.$$
\end{remark}

Our fixed extremal K3 surface will have an elliptic fibration $S \to \P^1$ with singular fibers: $2$ $I_8$ and $4$ $I_2$. It has exactly $8$ sections (it is \cite[Table 2 (25)]{SZ01}). Our configuration $\AR_0$ is the union of the $24$ $(-2)$-curves from the singular fibers and the $8$ sections. We now give a construction of the elliptic fibration $S \to \P^1$ and their sections.

Let us consider in $\P^2$ the pencil of cubics $$\{ 4\nu xyz + \mu(x-y)(z^2-xy)=0 \colon [\nu, \mu] \in \P^1 \} .$$ It has base points at $[0,0,1], [0,1,0], [1,0,0], [1,1,0]$. Let $S_0 \to \P^1$ be the corresponding elliptic fibration, where $S_0$ is a blow-up of $\P^2$ nine times. The singular fibers are over $[1,0]$ ($I_8$ fiber), $[0,1]$ ($I_2$ fiber), $[1,1]$ and $[1,-1]$ ($I_1$ fibers, with nodes at $[1,-1,1]$ and $[-1,1,1]$). We consider in $S_0$ the following smooth rational curves:

\begin{enumerate}
\item $L_1$, $L_2$, $L_3$ proper transforms of $\{z=0\}$, $\{y=0\}$, $\{x=0\}$ respectively.

\item $E_1$, $E_2$, $E_3$ over $[0,1,0]$, where $E_3$ is a section.

\item $E_4$, $E_5$ over $[0,0,1]$, where $E_5$ is a section.

\item $E_6$, $E_7$, $E_8$ over $[1,0,0]$, where $E_8$ is a section.

\item $E_9$ over $[1,1,0]$.

\item $L$ and $C$ proper transforms of $\{x-y=0\}$ and $\{z^2-xy=0\}$ respectively.

\item $M$ and $N$ proper transforms of $\{x+y=0\}$ and $\{z^2+xy=0\}$ respectively.
\end{enumerate}

The $I_8$ fiber in $S_0$ is formed by the cycle $L_1, E_1, E_2, L_3, E_4, L_2, E_7, E_6$; the fiber $I_2$ is formed by $L$ and $C$. The two $I_1$ fibers are denoted by $G_1$ and $G_2$. The four sections of $S_0 \to \P^1$ are $E_3$, $E_5$, $E_8$, $E_9$. The curves $M$ and $N$ are bisections, each of them passing through the node of $G_1$ and the node of $G_2$. Thus $G_i$, $M$, $N$ form a simple point of multiplicity $4$ at the node of $G_i$ for $i=1,2$.

The K3 surface $S$ is obtained by resolving the base change of degree $2$ of $S_0 \to \P^1$, which is ramified over $G_1$ and $G_2$. The resolution happens over the nodes of $G_1$ and $G_2$. Let $f \colon S \to S_0$ be the corresponding double cover composed with the resolution. The pre-image of $I_8$ by $f$ is equal to two $I_8$, the pre-image of $I_2$ is equal to two $I_2$, and the pre-images of each $I_1$ is equal to one $I_2$ fiber. The elliptic fibration $S \to \P^1$ has four sections from the pre-images of $E_3$, $E_5$, $E_8$, $E_9$, and four more sections from the pre-images of $M$ and $N$. We now list and name the pre-images of these distinguished curves, and in Figure \ref{f1} we show their intersections.

\begin{enumerate}
\item $f^{-1}(I_8)= \sum_{i=1}^8 F_i + \sum_{i=9}^{16} F_i$.

\item $f^{-1}(I_2)= (B_1+B_2) + (B_3+B_4)$.

\item $f^{-1}(G_1)= C_1 + C_2$ and $f^{-1}(G_2)= C_3 + C_4$.

\item $f^{-1}(E_3)=A_4$, $f^{-1}(E_5)=A_3$, $f^{-1}(E_8)=A_1$ and $f^{-1}(E_9)=A_2$.

\item $f^{-1}(N)= D_1 + D_2$ and $f^{-1}(M)= D_3 + D_4$.
\end{enumerate}

\begin{figure}[htbp]
\centering
\includegraphics[width=12.7cm]{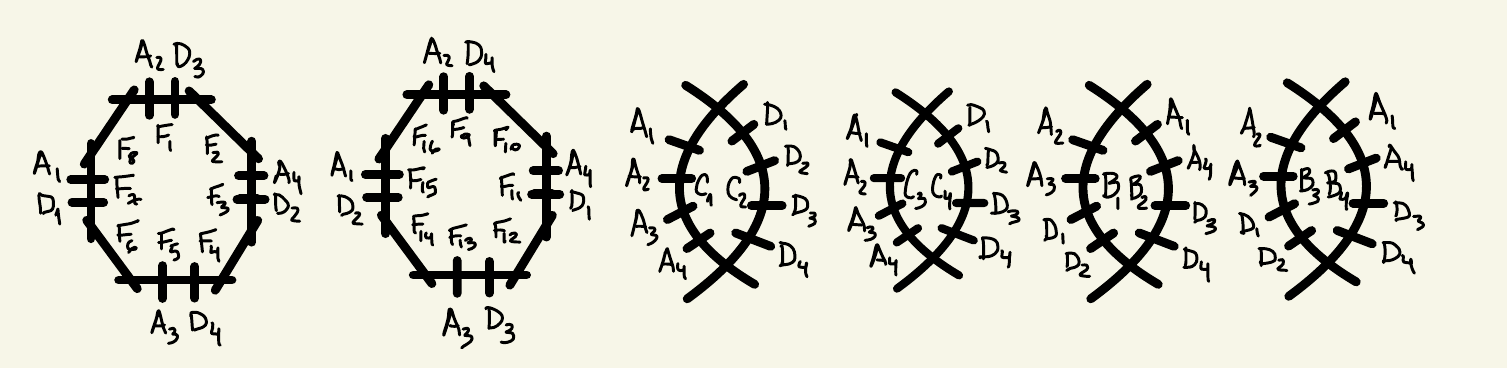}
\caption{The main configuration $\AR_0$}
\label{f1}
\end{figure}

\vspace{0.5cm}

Below we will show the existence of one particular surface $X$ with two Wahl singularities and some rational double points, which have no-local-to-global obstructions to deform, $K$ is ample, and $2 \leq K^2 \leq 9$ (we treated $K^2=1$ in Section \ref{s2}). We will describe various particular properties in each case to give some general picture of what can be expected. At the end we provide the necessary data to build more examples, but without further details.

In the figures below: the thicker dots are the points where we blow-up, the x means a point where more than one blow-up will be applied, and the dashed segments represent $(-1)$-curves.

\subsection{$K^2=2$} Let us consider the subconfiguration of $\AR_0$ given by the curves $C_1,C_2,B_1,B_3,A_2,D_1$. Let $\pi \colon \tilde{X} \to S$ be the blow-up indicated in Figure \ref{f2}.

\begin{figure}[htbp]
\centering
\includegraphics[width=12.5cm]{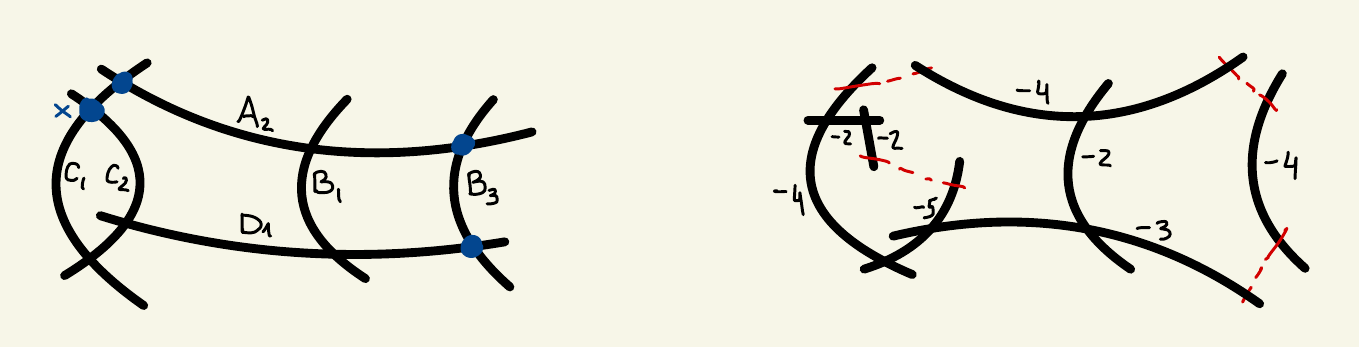}
\caption{Configuration for $K^2=2$}
\label{f2}
\end{figure}

We have two Wahl chains $[4]$ and $[4,2,3,5,4,2,2]$. Let $\phi \colon \tilde{X} \to X$ be the contraction of these two Wahl chains, and the Du Val chains $F_8$, $F_2+F_3+F_4+F_5+F_6$, $F_{10}$, and $F_{12}+F_{13}+F_{14}+F_{15}+F_{16}$. By Proposition \ref{nef}, the canonical class of $X$ is ample, and $X$ has singularities $$\frac{1}{2^2}(1,1) \ \ \ \ \frac{1}{27^2}(1,27 \cdot 8-1) \ \ \ \ 2 \times \frac{1}{2}(1,1) \ \ \ \ 2 \times \frac{1}{6}(1,5).$$

The intersection matrix of this subconfiguration of $\AR_0$ is as follows
\[\left(\begin{array}{*{6}c}
    -2 & 2 & 0 & 0 & 1 & 0\\
    2 & -2 & 0 & 0 & 0 & 1\\
    0 & 0 & -2 & 0 & 1 & 1\\
    0 & 0 & 0 & -2 & 1 & 1\\
    1 & 0 & 1 & 1 & -2 & 0\\
    0 & 1 & 1 & 1 & 0 & -2
\end{array}\right).\]
Its determinant is $-28$.

By Proposition \ref{obstruction}, we have that $X$ has no-local-to-global obstructions, and so it represents a smooth point in the KSBA compactification of dimension $16$. A general $\Q$-Gorenstein smoothing has canonical class ample. They are simply-connected surfaces, because gcd$(2,27)=1$ and there is a $\P^1$ joining the ends of the Wahl chains (see Proposition \ref{pi1}). We have $K^2=2$, and $p_g=1$. There are also $\Q$-Gorenstein smoothings producing simply-connected, $K^2=2$, $p_g=1$ surfaces with any Du Val subconfigurations from the minimal resolution of  $2 \times \frac{1}{2}(1,1)+ 2 \times \frac{1}{6}(1,5)$.

The two Wahl singularities define two divisors in this KSBA compactification, whose general member is a surface with one Wahl singularity and smooth elsewhere. If we keep the singularity from $[4]$, then we obtain that the minimal resolution is a $K^2=1$ surface of general type; if we keep the singularity from $[4,2,3,5,4,2,2]$, then we obtain a blow-up five times of a K3 surface as minimal resolution. That can be shown using flips as in \cite{U16b}.

\subsection{$K^2=3$} We now consider the subconfiguration of $\AR_0$ given by the curves $F_1, C_1, C_2, C_3, B_1, A_2, A_3, D_3$. Let $\pi \colon \tilde{X} \to S$ be the blow-up indicated in Figure \ref{f3}. Those blow-ups produce two Wahl chains $[2,3,5,3]$ and $[2,2,5,5,2,2,4]$. Let $\phi \colon \tilde{X} \to X$ be the contraction of these two Wahl chains, and the Du Val chains $F_{10}+ F_{11}+ F_{12}$, $F_{14}+ F_{15}+ F_{16}$, $F_3+F_4$, and $F_6 + F_7$. By Proposition \ref{nef}, the canonical class of $X$ is ample, and $X$ has singularities $$\frac{1}{8^2}(1,8 \cdot 3-1) \ \ \ \ \frac{1}{23^2}(1,23 \cdot 7-1) \ \ \ \ 2 \times \frac{1}{4}(1,3) \ \ \ \ 2 \times \frac{1}{3}(1,2).$$

\begin{figure}[htbp]
\centering
\includegraphics[width=12.5cm]{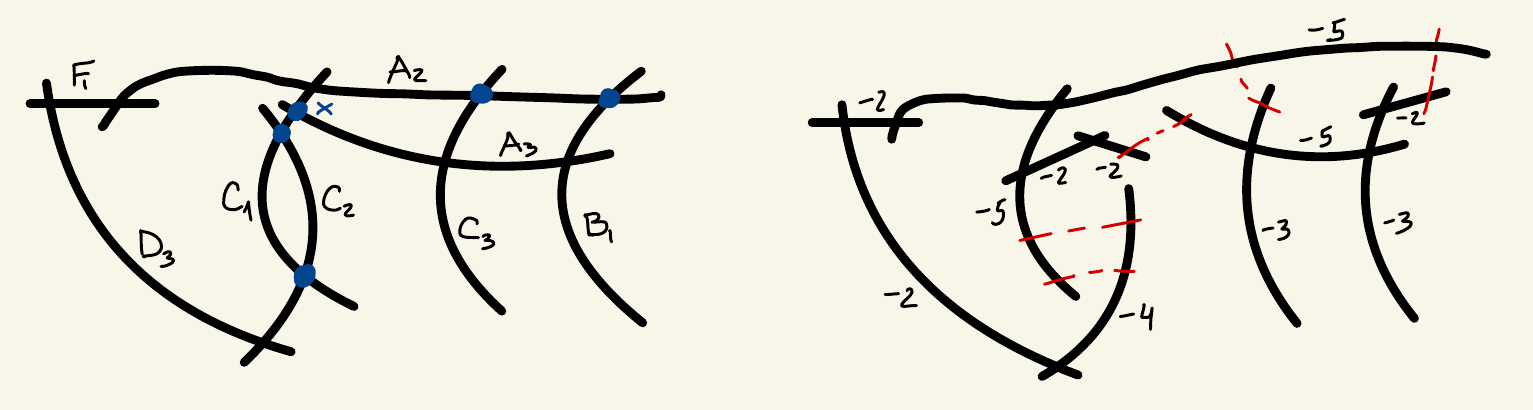}
\caption{Configuration for $K^2=3$}
\label{f3}
\end{figure}

The intersection matrix of this subconfiguration of $\AR_0$ is as follows
\[\left(\begin{array}{*{8}c}
    -2 & 0 & 0 & 0 & 0 & 1 & 0 & 1\\
    0 & -2 & 2 & 0 & 0 & 1 & 1 & 0\\
    0 & 2 & -2 & 0 & 0 & 0 & 0 & 1\\
    0 & 0 & 0 & -2 & 0 & 1 & 1 & 0\\
    0 & 0 & 0 & 0 & -2 & 1 & 1 & 0\\
    1 & 1 & 0 & 1 & 1 & -2 & 0 & 0\\
    0 & 1 & 0 & 1 & 1 & 0 & -2 & 0\\
    1 & 0 & 1 & 0 & 0 & 0 & 0 & -2
\end{array}\right).\]
Its determinant is $-76$.

By Proposition \ref{obstruction}, we have that $X$ has no-local-to-global obstructions, and so it represents a smooth point in the KSBA compactification of dimension $14$. A general $\Q$-Gorenstein smoothing has canonical class ample, $K^2=3$, and $p_g=1$. They are simply-connected by Proposition \ref{pi1} using $F_9$ and the $(-1)$-curve between $A_2$ and $B_1$. The two Wahl singularities define two divisors in this KSBA compactification, whose general member is a surface with one Wahl singularity and smooth elsewhere. In both cases the minimal resolution is a blown-up K3 surface. There are also $\Q$-Gorenstein smoothings producing simply-connected, $K^2=3$, $p_g=1$ surfaces with any Du Val subconfigurations from the minimal resolution of  $2 \times \frac{1}{4}(1,3) + 2 \times \frac{1}{3}(1,2)$.

\subsection{$K^2=4$} Take the configuration $F_1, F_2, F_3, C_1, C_2, C_3, B_1, A_2, A_3, D_2$. Let $\pi \colon \tilde{X} \to S$ be the blow-up indicated in Figure \ref{f4}. Hence we produce two Wahl chains $[2,2,2,3,3,5,3,5]$ and $[2,2,2,3,3,8,2,2,2,3,3,5]$. Let $\phi \colon \tilde{X} \to X$ be the contraction of these two Wahl chains, and the Du Val chains $B_4$, $F_{16}$, $F_{14}$, $F_6+F_7$, and $F_{10}+F_{11}+F_{12}$. By Proposition \ref{nef}, the canonical class of $X$ is ample, and $X$ has singularities $$\frac{1}{37^2}(1,37 \cdot 8-1) \ \ \ \ \frac{1}{129^2}(1,129 \cdot 28-1) \ \ \ \ 3 \times \frac{1}{2}(1,1) \ \ \ \ \frac{1}{3}(1,2) \ \ \ \ \frac{1}{4}(1,3).$$

\begin{figure}[htbp]
\centering
\includegraphics[width=12.5cm]{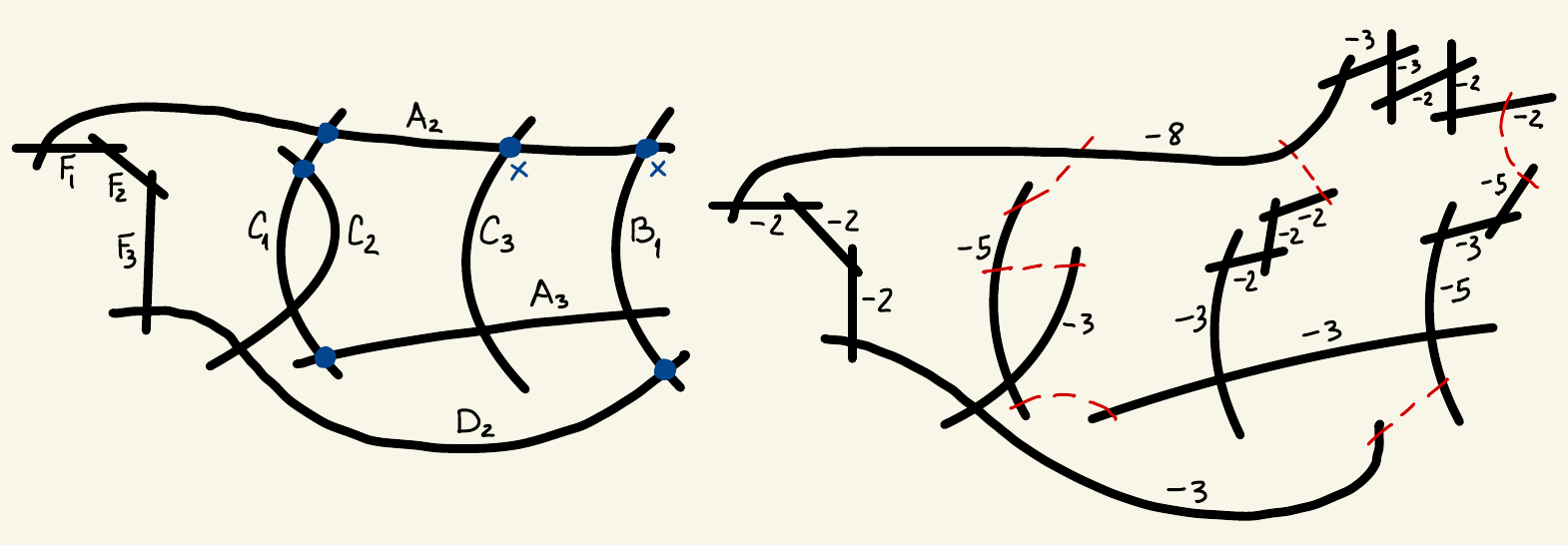}
\caption{Configuration for $K^2=4$}
\label{f4}
\end{figure}

The intersection matrix of this subconfiguration of $\AR_0$ is as follows
\[\left(\begin{array}{*{10}c}
    -2 & 1 & 0 & 0 & 0 & 0 & 0 & 1 & 0 & 0\\
    1 & -2 & 1 & 0 & 0 & 0 & 0 & 0 & 0 & 0\\
    0 & 1 & -2 & 0 & 0 & 0 & 0 & 0 & 0 & 1\\
    0 & 0 & 0 & -2 & 2 & 0 & 0 & 1 & 1 & 0\\
    0 & 0 & 0 & 2 & -2 & 0 & 0 & 0 & 0 & 1\\
    0 & 0 & 0 & 0 & 0 & -2 & 0 & 1 & 1 & 0\\
    0 & 0 & 0 & 0 & 0 & 0 & -2 & 1 & 1 & 1\\
    1 & 0 & 0 & 1 & 0 & 1 & 1 & -2 & 0 & 0\\
    0 & 0 & 0 & 1 & 0 & 1 & 1 & 0 & -2 & 0\\
    0 & 0 & 1 & 0 & 1 & 0 & 1 & 0 & 0 & -2
\end{array}\right).\]
Its determinant is $-136$.

By Proposition \ref{obstruction}, we have that $X$ has no-local-to-global obstructions, and so it represents a smooth point in the KSBA compactification of dimension $12$. A general $\Q$-Gorenstein smoothing has canonical class ample, which are simply-connected (direct from Proposition \ref{pi1}), $K^2=4$, and $p_g=1$. There are also $\Q$-Gorenstein smoothings producing simply-connected, $K^2=4$, $p_g=1$ surfaces with any Du Val subconfigurations from the minimal resolution of $3 \times \frac{1}{2}(1,1) + \frac{1}{3}(1,2) + \frac{1}{4}(1,3)$. The two Wahl singularities define two divisors in this KSBA compactification, whose general member is a surface with one Wahl singularity and smooth elsewhere. In both cases, the minimal resolution is a blown-up K3 surface.

\subsection{$K^2=5$} Let us take the subconfiguration $$F_1, F_2, F_3, C_1, C_2, C_3, B_1, A_3, A_4, D_1, D_2, D_3.$$ Let $\pi \colon \tilde{X} \to S$ be the blow-up indicated in Figure \ref{f5}. We now have the two Wahl chains $[2,2,2,5,5]$ and $$[2,2,2,3,3,3,8,2,2,2,3,3,3,5].$$ Let $\phi \colon \tilde{X} \to X$ be the contraction of these two Wahl chains, and the Du Val chains $F_6$, $F_{12}$, $F_{14}$, and $F_{9}+F_{10}+F_{16}$. By Proposition \ref{nef}, the canonical class of $X$ is ample, and $X$ has singularities $$\frac{1}{9^2}(1,9 \cdot 2-1) \ \ \ \ \frac{1}{337^2}(1,337 \cdot 73-1) \ \ \ \ 3 \times \frac{1}{2}(1,1) \ \ \ \ \frac{1}{4}(1,3).$$

\begin{figure}[htbp]
\centering
\includegraphics[width=9.3cm]{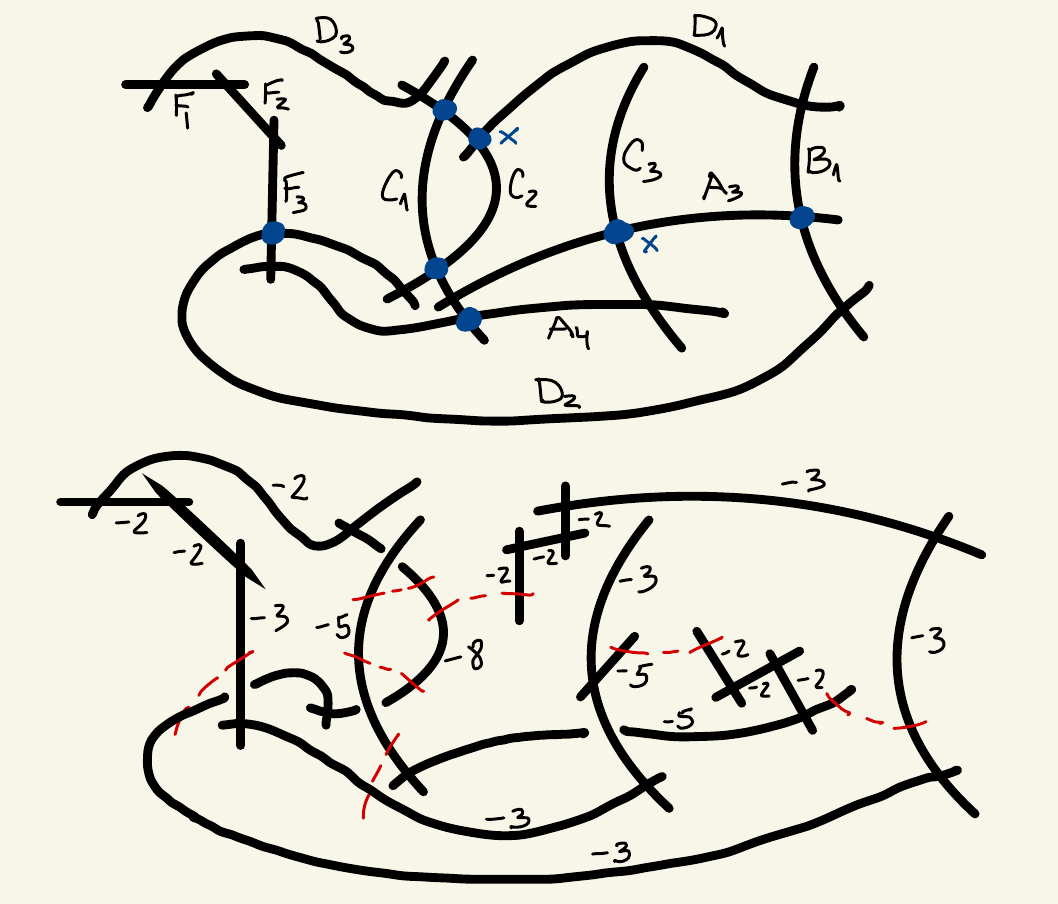}
\caption{Configuration for $K^2=5$}
\label{f5}
\end{figure}

The intersection matrix of this subconfiguration of $\AR_0$ is as follows
\[\left(\begin{array}{*{12}c}
    -2 & 1 & 0 & 0 & 0 & 0 & 0 & 0 & 0 & 0 & 0 & 1\\
    1 & -2 & 1 & 0 & 0 & 0 & 0 & 0 & 0 & 0 & 0 & 0\\
    0 & 1 & -2 & 0 & 0 & 0 & 0 & 0 & 1 & 0 & 1 & 0\\
    0 & 0 & 0 & -2 & 2 & 0 & 0 & 1 & 1 & 0 & 0 & 0\\
    0 & 0 & 0 & 2 & -2 & 0 & 0 & 0 & 0 & 1 & 1 & 1\\
    0 & 0 & 0 & 0 & 0 & -2 & 0 & 1 & 1 & 0 & 0 & 0\\
    0 & 0 & 0 & 0 & 0 & 0 & -2 & 1 & 0 & 1 & 1 & 0\\
    0 & 0 & 0 & 1 & 0 & 1 & 1 & -2 & 0 & 0 & 0 & 0\\
    0 & 0 & 1 & 1 & 0 & 1 & 0 & 0 & -2 & 0 & 0 & 0\\
    0 & 0 & 0 & 0 & 1 & 0 & 1 & 0 & 0 & -2 & 0 & 0\\
    0 & 0 & 1 & 0 & 1 & 0 & 1 & 0 & 0 & 0 & -2 & 0\\
    1 & 0 & 0 & 0 & 1 & 0 & 0 & 0 & 0 & 0 & 0 & -2
\end{array}\right).\]
Its determinant is $-264$.

By Proposition \ref{obstruction}, we have that $X$ has no-local-to-global obstructions, and so it represents a smooth point in the KSBA compactification of dimension $10$. A general $\Q$-Gorenstein smoothing has canonical class ample, which are simply-connected (direct from Proposition \ref{pi1}), $K^2=5$, and $p_g=1$. There are also $\Q$-Gorenstein smoothings producing simply-connected, $K^2=5$, $p_g=1$ surfaces with any Du Val subconfigurations from the minimal resolution of $3 \times \frac{1}{2}(1,1) + \frac{1}{4}(1,3)$. The two Wahl singularities define two divisors in this KSBA compactification, whose general member is a surface with one Wahl singularity and smooth elsewhere. The minimal resolution of $\frac{1}{9^2}(1,9 \cdot 2-1)$ gives an elliptic surface with a $[5,5,2,2,2]$ as in Proposition \ref{q=0pg=1}. The other comes from a K3 surface.

\subsection{$K^2=6$} Let us consider the subconfiguration $$F_1, F_2, F_3, F_4, F_5, C_1, C_2, C_3, B_1, B_4, A_1, A_2, A_3, D_3.$$ Let $\pi \colon \tilde{X} \to S$ be the blow-up indicated in Figure \ref{f6}. We now have the two Wahl chains $[2,5]$ and $$[2,2,2,2,2,3,4,3,9,2,2,2,2,3,3,2,3,7].$$ Let $\phi \colon \tilde{X} \to X$ be the contraction of these two Wahl chains, and the Du Val chains $F_{14}$, $F_{16}$, and $F_{10}+F_{11}+F_{12}$. By Proposition \ref{nef}, the canonical class of $X$ is ample, and $X$ has singularities $$\frac{1}{3^2}(1,3 \cdot 1-1) \ \ \ \ \frac{1}{829^2}(1,829 \cdot 126-1) \ \ \ \ 2 \times \frac{1}{2}(1,1) \ \ \ \ \frac{1}{4}(1,3).$$

\begin{figure}[htbp]
\centering
\includegraphics[width=11cm]{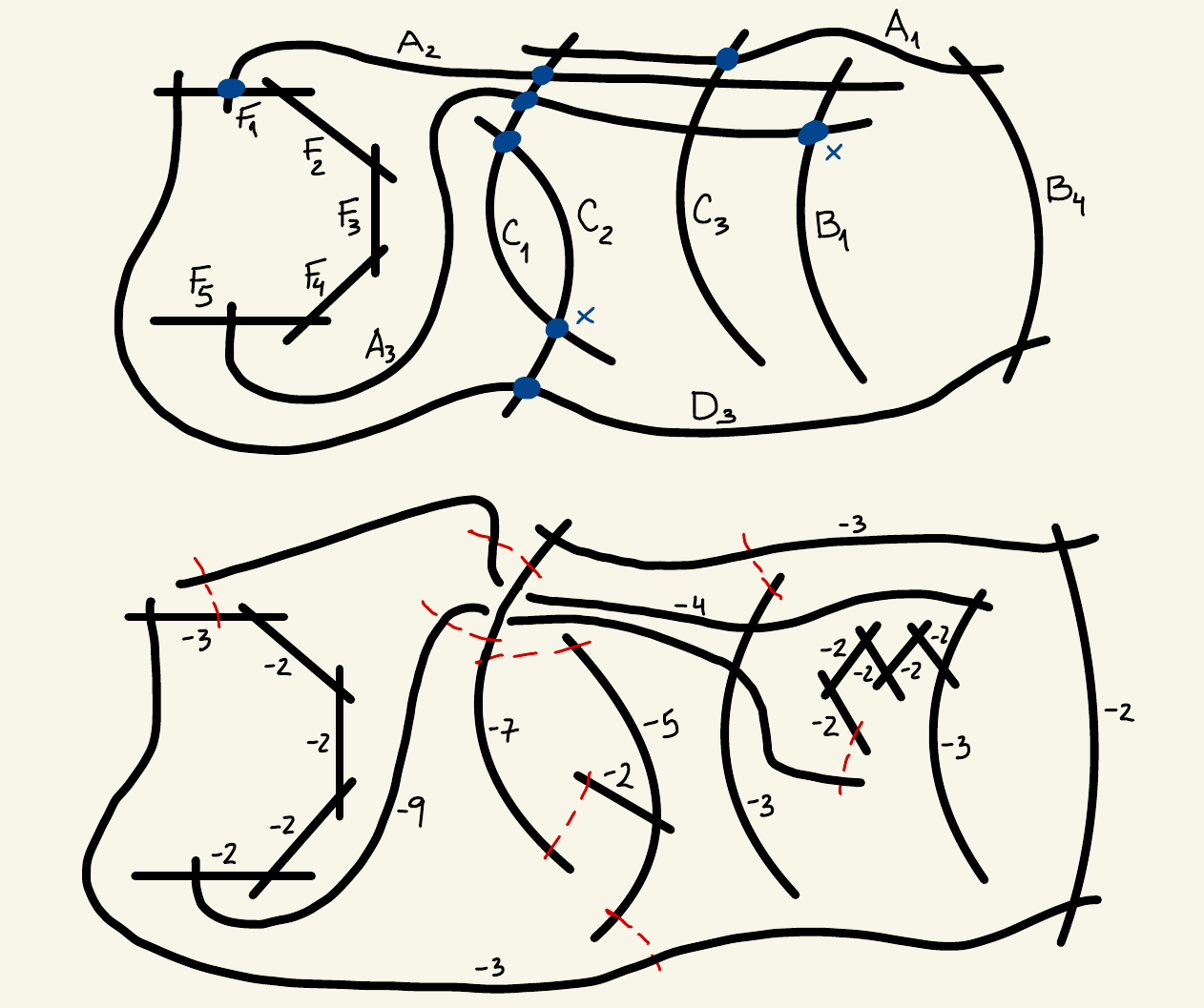}
\caption{Configuration for $K^2=6$}
\label{f6}
\end{figure}

The intersection matrix of this subconfiguration of $\AR_0$ is as follows
\setlength{\templength}{\arraycolsep}
\setlength{\arraycolsep}{2.5pt}
\[\left(\begin{array}{*{14}c}
    -2 & 1 & 0 & 0 & 0 & 0 & 0 & 0 & 0 & 0 & 0 & 1 & 0 & 1\\
    1 & -2 & 1 & 0 & 0 & 0 & 0 & 0 & 0 & 0 & 0 & 0 & 0 & 0\\
    0 & 1 & -2 & 1 & 0 & 0 & 0 & 0 & 0 & 0 & 0 & 0 & 0 & 0\\
    0 & 0 & 1 & -2 & 1 & 0 & 0 & 0 & 0 & 0 & 0 & 0 & 0 & 0\\
    0 & 0 & 0 & 1 & -2 & 0 & 0 & 0 & 0 & 0 & 0 & 0 & 1 & 0\\
    0 & 0 & 0 & 0 & 0 & -2 & 2 & 0 & 0 & 0 & 1 & 1 & 1 & 0\\
    0 & 0 & 0 & 0 & 0 & 2 & -2 & 0 & 0 & 0 & 0 & 0 & 0 & 1\\
    0 & 0 & 0 & 0 & 0 & 0 & 0 & -2 & 0 & 0 & 1 & 1 & 1 & 0\\
    0 & 0 & 0 & 0 & 0 & 0 & 0 & 0 & -2 & 0 & 0 & 1 & 1 & 0\\
    0 & 0 & 0 & 0 & 0 & 0 & 0 & 0 & 0 & -2 & 1 & 0 & 0 & 1\\
    0 & 0 & 0 & 0 & 0 & 1 & 0 & 1 & 0 & 1 & -2 & 0 & 0 & 0\\
    1 & 0 & 0 & 0 & 0 & 1 & 0 & 1 & 1 & 0 & 0 & -2 & 0 & 0\\
    0 & 0 & 0 & 0 & 1 & 1 & 0 & 1 & 1 & 0 & 0 & 0 & -2 & 0\\
    1 & 0 & 0 & 0 & 0 & 0 & 1 & 0 & 0 & 1 & 0 & 0 & 0 & -2
\end{array}\right).\]
\setlength{\arraycolsep}{\templength}
Its determinant is $-192$.

By Proposition \ref{obstruction}, we have that $X$ has no-local-to-global obstructions, and so it represents a smooth point in the KSBA compactification of dimension $8$. A general $\Q$-Gorenstein smoothing has canonical class ample, which are simply-connected (Proposition \ref{pi1}), $K^2=6$, and $p_g=1$. 

The two Wahl singularities define two divisors in this KSBA compactification, for $[5,2]$ the minimal resolution is a blown-up surface of general type with $K^2\leq 4$, for the other chain we obtain a blown-up K3 surface. There are also $\Q$-Gorenstein smoothings producing simply-connected, $K^2=6$, $p_g=1$ surfaces with any Du Val subconfigurations from the minimal resolution of $2 \times \frac{1}{2}(1,1) + \frac{1}{4}(1,3)$.

\subsection{$K^2=7$} Let us consider the subconfiguration of $\AR_0$ given by the curves $$F_1, F_2, F_3, F_7, F_8, C_1, C_2, C_3, B_1, F_9, F_{10}, F_{11}, A_1, A_2, D_1, D_2.$$ Let $\pi \colon \tilde{X} \to S$ be the blow-up indicated in Figure \ref{f7}. Thus we have the two Wahl chains $[4]$ and $$[2,2,2,3,2,3,3,4,2,2,2,5,6,2,3,3,4,5].$$ Let $\phi \colon \tilde{X} \to X$ be the contraction of these two Wahl chains, and the Du Val chains $F_5$, and $F_{13}+F_{14}$. 

By Proposition \ref{nef}, the canonical class of $X$ is ample, and $X$ has singularities $$\frac{1}{4}(1,1) \ \ \ \ \frac{1}{2168^2}(1,2168 \cdot 459-1) \ \ \ \ \frac{1}{2}(1,1) \ \ \ \ \frac{1}{3}(1,2).$$

\begin{figure}[htbp]
\centering
\includegraphics[width=10.5cm]{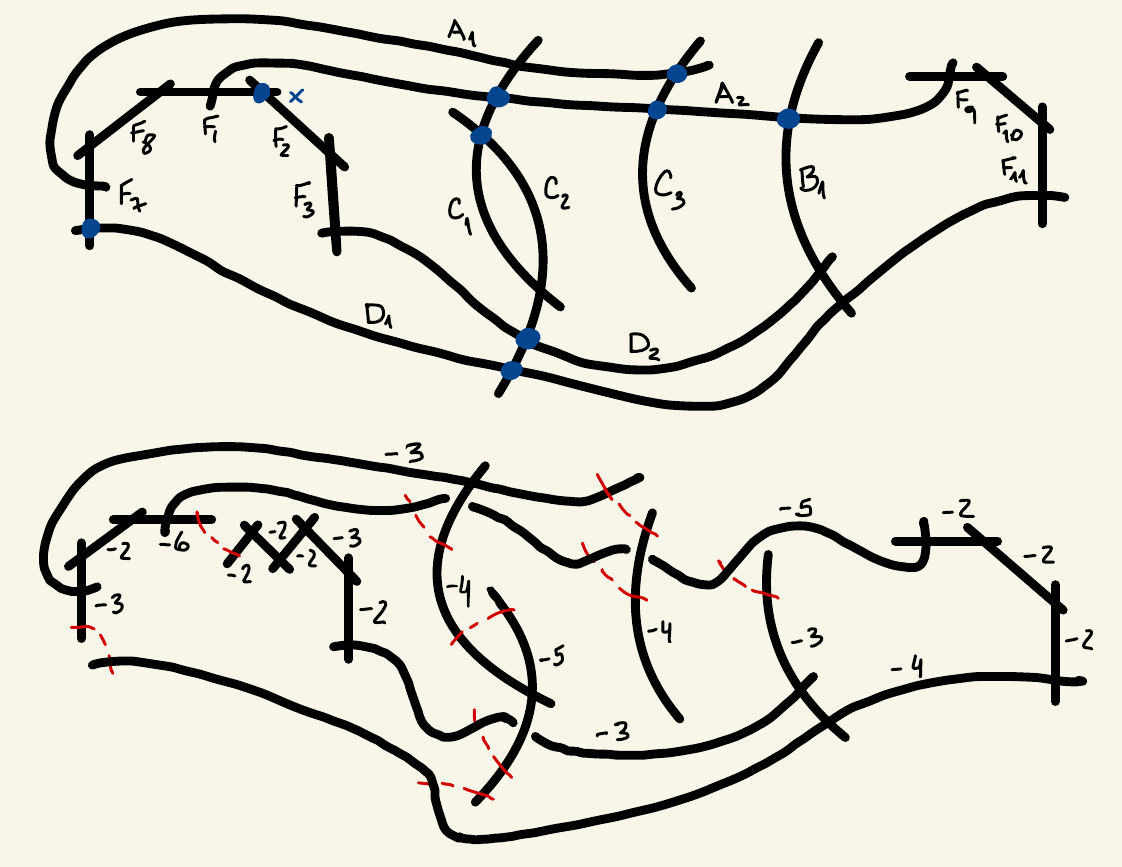}
\caption{Configuration for $K^2=7$}
\label{f7}
\end{figure}

The intersection matrix of this subconfiguration of $\AR_0$ is as follows
\setlength{\templength}{\arraycolsep}
\setlength{\arraycolsep}{2.5pt}
\[\left(\begin{array}{*{16}c}
    -2 & 1 & 0 & 0 & 1 & 0 & 0 & 0 & 0 & 0 & 0 & 0 & 0 & 1 & 0 & 0\\
    1 & -2 & 1 & 0 & 0 & 0 & 0 & 0 & 0 & 0 & 0 & 0 & 0 & 0 & 0 & 0\\
    0 & 1 & -2 & 0 & 0 & 0 & 0 & 0 & 0 & 0 & 0 & 0 & 0 & 0 & 0 & 1\\
    0 & 0 & 0 & -2 & 1 & 0 & 0 & 0 & 0 & 0 & 0 & 0 & 1 & 0 & 1 & 0\\
    1 & 0 & 0 & 1 & -2 & 0 & 0 & 0 & 0 & 0 & 0 & 0 & 0 & 0 & 0 & 0\\
    0 & 0 & 0 & 0 & 0 & -2 & 2 & 0 & 0 & 0 & 0 & 0 & 1 & 1 & 0 & 0\\
    0 & 0 & 0 & 0 & 0 & 2 & -2 & 0 & 0 & 0 & 0 & 0 & 0 & 0 & 1 & 1\\
    0 & 0 & 0 & 0 & 0 & 0 & 0 & -2 & 0 & 0 & 0 & 0 & 1 & 1 & 0 & 0\\
    0 & 0 & 0 & 0 & 0 & 0 & 0 & 0 & -2 & 0 & 0 & 0 & 0 & 1 & 1 & 1\\
    0 & 0 & 0 & 0 & 0 & 0 & 0 & 0 & 0 & -2 & 1 & 0 & 0 & 1 & 0 & 0\\
    0 & 0 & 0 & 0 & 0 & 0 & 0 & 0 & 0 & 1 & -2 & 1 & 0 & 0 & 0 & 0\\
    0 & 0 & 0 & 0 & 0 & 0 & 0 & 0 & 0 & 0 & 1 & -2 & 0 & 0 & 1 & 0\\
    0 & 0 & 0 & 1 & 0 & 1 & 0 & 1 & 0 & 0 & 0 & 0 & -2 & 0 & 0 & 0\\
    1 & 0 & 0 & 0 & 0 & 1 & 0 & 1 & 1 & 1 & 0 & 0 & 0 & -2 & 0 & 0\\
    0 & 0 & 0 & 1 & 0 & 0 & 1 & 0 & 1 & 0 & 0 & 1 & 0 & 0 & -2 & 0\\
    0 & 0 & 1 & 0 & 0 & 0 & 1 & 0 & 1 & 0 & 0 & 0 & 0 & 0 & 0 & -2
\end{array}\right).\]
\setlength{\arraycolsep}{\templength}
Its determinant is $-104$.

By Proposition \ref{obstruction}, we have that $X$ has no-local-to-global obstructions, and so it represents a smooth point in the KSBA compactification of dimension $6$. A general $\Q$-Gorenstein smoothing has canonical class ample, $K^2=7$, and $p_g=1$. Since gcd$(2168,4)=4 \neq 1$, we need a bit more in the argument to show simply-connectedness. We use the relations in \cite[p.20]{M61} on the exponent of a loop around an exceptional curve in $\exc(\phi)$ with respect to the generator $\gamma$ around the last $(-2)$-curve in the Wahl chain. Since $F_5=\P^1$ intersects transversally $F_3$ at one point, we have that $\gamma^9=1$. But because of $F_{12}$, we have that $\gamma^{307}=1$, and so $\gamma$ is trivial. But then a loop around the $[4]$ chain is also trivial, and we conclude using Van-Kampen that the smooth fiber is simply-connected (as in \cite[Theorem 3]{LP07}). The two Wahl singularities define two divisors in this KSBA compactification, for $[4]$ the minimal resolution is a surface of general type with $K^2=6$, for the other chain we obtain a blown-up K3 surface. There are also $\Q$-Gorenstein smoothings producing simply-connected, $K^2=7$, $p_g=1$ surfaces with any Du Val subconfigurations from the minimal resolution of $\frac{1}{2}(1,1) + \frac{1}{3}(1,2)$.

\subsection{$K^2=8$} Let us consider the subconfiguration of $\AR_0$ given by the curves $$F_1, F_2, F_3, F_7, F_8, C_1, C_2, C_4, B_1, F_{11}, F_{12}, F_{13}, F_{14}, F_{15}, A_1, A_2, D_1, D_2.$$ Let $\pi \colon \tilde{X} \to S$ be the blow-up indicated in Figure \ref{f8}. We have the two Wahl chains $[4]$ and $$[2,2,2,2,5,4,3,3,9,2,2,2,2,3,3,3,2,3,2,2,6].$$ Let $\phi \colon \tilde{X} \to X$ be the contraction of these two Wahl chains, and the $(-2)$-curve $F_{5}$. By Proposition \ref{nef}, the canonical class of $X$ is ample, and $X$ has singularities $$\frac{1}{4}(1,1) \ \ \ \ \frac{1}{3767^2}(1,3767 \cdot 715-1) \ \ \ \  \frac{1}{2}(1,1).$$

\begin{figure}[htbp]
\centering
\includegraphics[width=11.5cm]{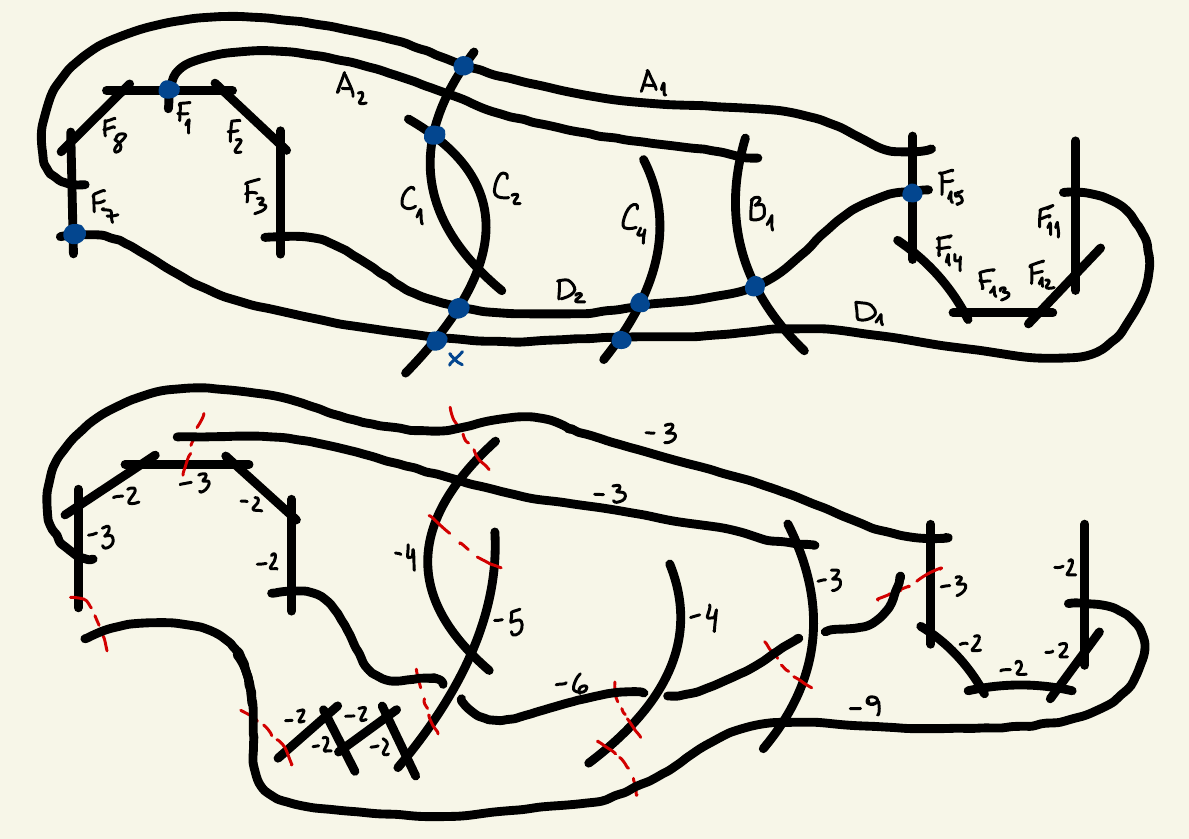}
\caption{Configuration for $K^2=8$}
\label{f8}
\end{figure}

The intersection matrix of this subconfiguration of $\AR_0$ is as follows
\begin{footnotesize}
\setlength{\templength}{\arraycolsep}
\setlength{\arraycolsep}{2.5pt}
\[\left(\begin{array}{*{18}c}
    -2 & 1 & 0 & 0 & 1 & 0 & 0 & 0 & 0 & 0 & 0 & 0 & 0 & 0 & 0 & 1 & 0 & 0\\
    1 & -2 & 1 & 0 & 0 & 0 & 0 & 0 & 0 & 0 & 0 & 0 & 0 & 0 & 0 & 0 & 0 & 0\\
    0 & 1 & -2 & 0 & 0 & 0 & 0 & 0 & 0 & 0 & 0 & 0 & 0 & 0 & 0 & 0 & 0 & 1\\
    0 & 0 & 0 & -2 & 1 & 0 & 0 & 0 & 0 & 0 & 0 & 0 & 0 & 0 & 1 & 0 & 1 & 0\\
    1 & 0 & 0 & 1 & -2 & 0 & 0 & 0 & 0 & 0 & 0 & 0 & 0 & 0 & 0 & 0 & 0 & 0\\
    0 & 0 & 0 & 0 & 0 & -2 & 2 & 0 & 0 & 0 & 0 & 0 & 0 & 0 & 1 & 1 & 0 & 0\\
    0 & 0 & 0 & 0 & 0 & 2 & -2 & 0 & 0 & 0 & 0 & 0 & 0 & 0 & 0 & 0 & 1 & 1\\
    0 & 0 & 0 & 0 & 0 & 0 & 0 & -2 & 0 & 0 & 0 & 0 & 0 & 0 & 0 & 0 & 1 & 1\\
    0 & 0 & 0 & 0 & 0 & 0 & 0 & 0 & -2 & 0 & 0 & 0 & 0 & 0 & 0 & 1 & 1 & 1\\
    0 & 0 & 0 & 0 & 0 & 0 & 0 & 0 & 0 & -2 & 1 & 0 & 0 & 0 & 0 & 0 & 1 & 0\\
    0 & 0 & 0 & 0 & 0 & 0 & 0 & 0 & 0 & 1 & -2 & 1 & 0 & 0 & 0 & 0 & 0 & 0\\
    0 & 0 & 0 & 0 & 0 & 0 & 0 & 0 & 0 & 0 & 1 & -2 & 1 & 0 & 0 & 0 & 0 & 0\\
    0 & 0 & 0 & 0 & 0 & 0 & 0 & 0 & 0 & 0 & 0 & 1 & -2 & 1 & 0 & 0 & 0 & 0\\
    0 & 0 & 0 & 0 & 0 & 0 & 0 & 0 & 0 & 0 & 0 & 0 & 1 & -2 & 1 & 0 & 0 & 1\\
    0 & 0 & 0 & 1 & 0 & 1 & 0 & 0 & 0 & 0 & 0 & 0 & 0 & 1 & -2 & 0 & 0 & 0\\
    1 & 0 & 0 & 0 & 0 & 1 & 0 & 0 & 1 & 0 & 0 & 0 & 0 & 0 & 0 & -2 & 0 & 0\\
    0 & 0 & 0 & 1 & 0 & 0 & 1 & 1 & 1 & 1 & 0 & 0 & 0 & 0 & 0 & 0 & -2 & 0\\
    0 & 0 & 1 & 0 & 0 & 0 & 1 & 1 & 1 & 0 & 0 & 0 & 0 & 1 & 0 & 0 & 0 & -2
\end{array}\right).\]
\setlength{\arraycolsep}{\templength}
\end{footnotesize}
Its determinant is $-160$.

By Proposition \ref{obstruction}, we have that $X$ has no-local-to-global obstructions, and so it represents a smooth point in the KSBA compactification of dimension $4$. A general $\Q$-Gorenstein smoothing has canonical class ample, which are simply-connected (direct from Proposition \ref{pi1} using the $\P^1$ from $D_2 \cap C_4$), $K^2=8$, and $p_g=1$. The two Wahl singularities define two divisors in this KSBA compactification, whose general member is a surface with one Wahl singularity and smooth elsewhere. If we keep the $1/4(1,1)$ and smooth the rest, then we obtain a surface whose minimal resolution is a minimal surface of general type with $K^2=7$ (another way to produce a $K^2=7$ example as the previous one); If we keep the other Wahl singularity, then we obtain a blown-up K3 surface. There are also $\Q$-Gorenstein smoothings producing simply-connected, $K^2=8$, $p_g=1$ surfaces with a $(-2)$-curve inside.

\subsection{$K^2=9$} We finally take the configuration
$$F_1, F_2, F_3, F_4, F_5, F_6, F_7, C_1, C_2, B_1, F_9, F_{10}, F_{11}, F_{12}, F_{13}, A_1, A_2, A_3, D_1, D_2.$$ Let $\pi \colon \tilde{X} \to S$ be the blow-up indicated in Figure \ref{f9}. Hence we produce two Wahl chains $[2,2,2,3,2,6,5]$ and $$[2,2,2,3,2,3,3,2,3,2,5,7,2,2,3,2,2,4,4,3,4,5].$$ Let $\phi \colon \tilde{X} \to X$ be the contraction of these two Wahl chains. By Proposition \ref{nef}, the canonical class of $X$ is ample, and $X$ has singularities $$\frac{1}{19^2}(1,19 \cdot 4-1) \ \ \ \ \frac{1}{10730^2}(1,10730 \cdot 2271-1).$$

\begin{figure}[htbp]
\centering
\includegraphics[width=11cm]{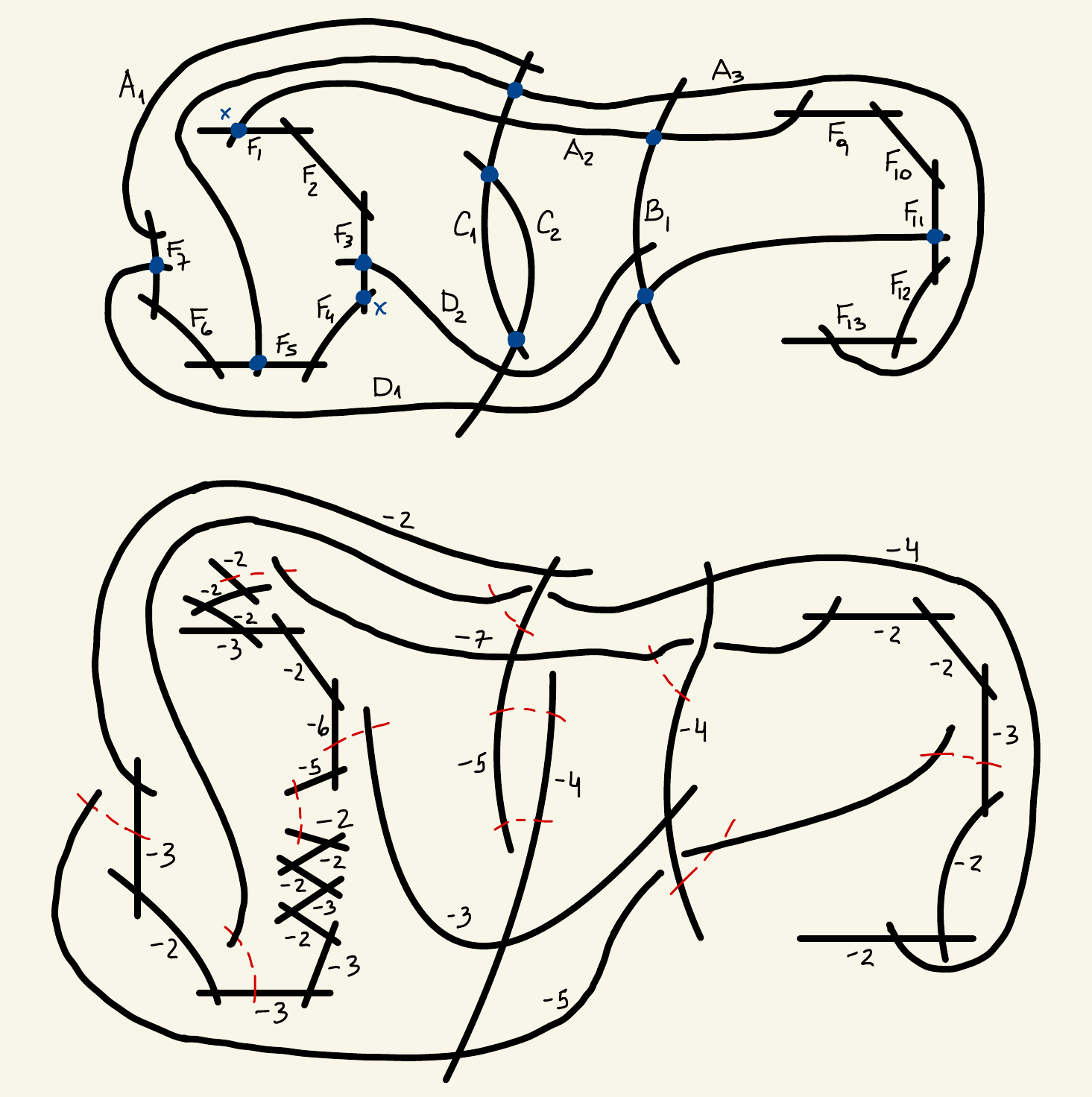}
\caption{Configuration for $K^2=9$}
\label{f9}
\end{figure}

The intersection matrix of this subconfiguration of $\AR_0$ is as follows
\begin{footnotesize}
\setlength{\templength}{\arraycolsep}
\setlength{\arraycolsep}{2.0pt}
\[\left(\begin{array}{*{20}c}
    -2 & 1 & 0 & 0 & 0 & 0 & 0 & 0 & 0 & 0 & 0 & 0 & 0 & 0 & 0 & 0 & 1 & 0 & 0 & 0\\
    1 & -2 & 1 & 0 & 0 & 0 & 0 & 0 & 0 & 0 & 0 & 0 & 0 & 0 & 0 & 0 & 0 & 0 & 0 & 0\\
    0 & 1 & -2 & 1 & 0 & 0 & 0 & 0 & 0 & 0 & 0 & 0 & 0 & 0 & 0 & 0 & 0 & 0 & 0 & 1\\
    0 & 0 & 1 & -2 & 1 & 0 & 0 & 0 & 0 & 0 & 0 & 0 & 0 & 0 & 0 & 0 & 0 & 0 & 0 & 0\\
    0 & 0 & 0 & 1 & -2 & 1 & 0 & 0 & 0 & 0 & 0 & 0 & 0 & 0 & 0 & 0 & 0 & 1 & 0 & 0\\
    0 & 0 & 0 & 0 & 1 & -2 & 1 & 0 & 0 & 0 & 0 & 0 & 0 & 0 & 0 & 0 & 0 & 0 & 0 & 0\\
    0 & 0 & 0 & 0 & 0 & 1 & -2 & 0 & 0 & 0 & 0 & 0 & 0 & 0 & 0 & 1 & 0 & 0 & 1 & 0\\
    0 & 0 & 0 & 0 & 0 & 0 & 0 & -2 & 2 & 0 & 0 & 0 & 0 & 0 & 0 & 1 & 1 & 1 & 0 & 0\\
    0 & 0 & 0 & 0 & 0 & 0 & 0 & 2 & -2 & 0 & 0 & 0 & 0 & 0 & 0 & 0 & 0 & 0 & 1 & 1\\
    0 & 0 & 0 & 0 & 0 & 0 & 0 & 0 & 0 & -2 & 0 & 0 & 0 & 0 & 0 & 0 & 1 & 1 & 1 & 1\\
    0 & 0 & 0 & 0 & 0 & 0 & 0 & 0 & 0 & 0 & -2 & 1 & 0 & 0 & 0 & 0 & 1 & 0 & 0 & 0\\
    0 & 0 & 0 & 0 & 0 & 0 & 0 & 0 & 0 & 0 & 1 & -2 & 1 & 0 & 0 & 0 & 0 & 0 & 0 & 0\\
    0 & 0 & 0 & 0 & 0 & 0 & 0 & 0 & 0 & 0 & 0 & 1 & -2 & 1 & 0 & 0 & 0 & 0 & 1 & 0\\
    0 & 0 & 0 & 0 & 0 & 0 & 0 & 0 & 0 & 0 & 0 & 0 & 1 & -2 & 1 & 0 & 0 & 0 & 0 & 0\\
    0 & 0 & 0 & 0 & 0 & 0 & 0 & 0 & 0 & 0 & 0 & 0 & 0 & 1 & -2 & 0 & 0 & 1 & 0 & 0\\
    0 & 0 & 0 & 0 & 0 & 0 & 1 & 1 & 0 & 0 & 0 & 0 & 0 & 0 & 0 & -2 & 0 & 0 & 0 & 0\\
    1 & 0 & 0 & 0 & 0 & 0 & 0 & 1 & 0 & 1 & 1 & 0 & 0 & 0 & 0 & 0 & -2 & 0 & 0 & 0\\
    0 & 0 & 0 & 0 & 1 & 0 & 0 & 1 & 0 & 1 & 0 & 0 & 0 & 0 & 1 & 0 & 0 & -2 & 0 & 0\\
    0 & 0 & 0 & 0 & 0 & 0 & 1 & 0 & 1 & 1 & 0 & 0 & 1 & 0 & 0 & 0 & 0 & 0 & -2 & 0\\
    0 & 0 & 1 & 0 & 0 & 0 & 0 & 0 & 1 & 1 & 0 & 0 & 0 & 0 & 0 & 0 & 0 & 0 & 0 & -2
\end{array}\right).\]
\setlength{\arraycolsep}{\templength}
\end{footnotesize}
Its determinant is $-16$.

By Proposition \ref{obstruction}, we have that $X$ has no-local-to-global obstructions, and so it represents a smooth point in the KSBA compactification of dimension $2$. A general $\Q$-Gorenstein smoothing has canonical class ample, which are simply-connected (Proposition \ref{pi1}), $K^2=9$, and $p_g=1$. The two Wahl singularities define two divisors in this KSBA compactification, whose general member is a surface with one Wahl singularity and smooth elsewhere. If we keep $[2,2,2,3,2,6,5]$ and smooth the other, then we obtain a surface of general type with $K^2=2$; for the other we obtain a blown-up K3 surface. We note that in this case we have $2$ singularities, and so this surface $X$ is maximally degenerated.

\subsection{More surfaces for each $2 \leq K^2 \leq 9$}

Below we give more examples (there are many more) of surfaces which have no-local-to-global obstructions to deform from the same configuration $\AR_0$. They are $\Q$-Gorenstein smoothable to simply connected surfaces with $p_g=1$ and $2 \leq K^2 \leq 9$. The description will be given in the following form:

\noindent
$\rule{12.5cm}{1.1pt}$
\noindent

$K^2$ - configuration $\AR$ - determinant of the intersection matrix of $\AR$ - points where blow ups happen and in the order shown - $(n_1,a_1):[b_1,\ldots,b_{\ell_1}]$ - $(n_2,a_2):[c_1,\ldots,c_{\ell_2}]$

\noindent
$\rule{12.5cm}{1.1pt}$

\bigskip
For instance, the data for the example in Figure \ref{f3} is:

\noindent
$\rule{12.5cm}{1.1pt}$
\noindent

$K^2=3$ - $\{F_1,~C_1,~C_2,~C_3,~B_1,~A_2,~A_3,~D_3\}$ - det$=-76$ - $C_1 \cap C_2$, $C_1 \cap C_2$, $[2,2,1] \times ~C_1 \cap A_3$, $C_3 \cap A_2$, $B_1 \cap A_2$ - $(8,3):[3,5,3,2]$, $(23,7):[4,2,2,5,5,2,2]$

\noindent
$\rule{12.5cm}{1.1pt}$

\bigskip
Each of these examples may have a particular property which we briefly discuss below the description. They always have two Wahl singularities, and they may have ADE configurations disjoint from $\AR$. 

\bigskip

\noindent
$\rule{12.5cm}{1.1pt}$
\noindent

\textbf{(2.1)} $K^2=2$ - $\{C_1,~\allowbreak C_2,~\allowbreak B_1,~\allowbreak A_2,~\allowbreak A_3,~\allowbreak D_1\}$ - $\operatorname{det}=-40$ - $~C_1 \cap C_2$, $~C_2 \cap D_1$, $[2,\allowbreak 2,1] \times ~A_2 \cap B_1$, $[2,1] \times ~A_3 \cap C_1$ - $(11,3) : [4,\allowbreak 5,\allowbreak 3,\allowbreak 2,\allowbreak 2]$, $(8,3) : [3,\allowbreak 5,\allowbreak 3,\allowbreak 2]$

\noindent
$\rule{12.5cm}{1.1pt}$
\noindent



\noindent
$\rule{12.5cm}{1.1pt}$
\noindent

\textbf{(2.2)} $K^2=2$ - $\{C_1,~\allowbreak C_2,~\allowbreak C_3,~\allowbreak B_1,~\allowbreak A_2,~\allowbreak A_3\}$ - $\operatorname{det}=-32$ - $~C_1 \cap A_3$, $~B_1 \cap A_2$, $[2,1] \times ~C_2 \cap C_1$, $[3,\allowbreak 2,\allowbreak 1,\allowbreak 3,\allowbreak 3,\allowbreak 2] \times ~A_2 \cap C_1$ - $(18,7) : [3,\allowbreak 3,\allowbreak 2,\allowbreak 6,\allowbreak 3,\allowbreak 2]$, $(18,7) : [3,\allowbreak 3,\allowbreak 2,\allowbreak 6,\allowbreak 3,\allowbreak 2]$

\noindent
$\rule{12.5cm}{1.1pt}$
\noindent

In example \textbf{(2.2)} we have $ [3,\allowbreak 3,\allowbreak 2,\allowbreak 6,\allowbreak 3,\allowbreak 2] - 1 -  [3,\allowbreak 3,\allowbreak 2,\allowbreak 6,\allowbreak 3,\allowbreak 2]$, and so it produces the T-singularity $\frac{1}{2 \cdot 18^2}(1, 2 \cdot 18 \cdot 7 -1)$.


\noindent
$\rule{12.5cm}{1.1pt}$
\noindent

\textbf{(2.3)} $K^2=2$ - $\{C_1,~\allowbreak C_2,~\allowbreak B_1,~\allowbreak B_3,~\allowbreak A_2,~\allowbreak D_1\}$ - $\operatorname{det}=-28$ - $~C_1 \cap A_2$, $~C_1 \cap C_2$, $[2,\allowbreak 2,1] \times ~B_3 \cap D_1$, $[3,\allowbreak 2,\allowbreak 2,\allowbreak 1,\allowbreak 4] \times ~B_1 \cap A_2$ - $(29,8) : [4,\allowbreak 3,\allowbreak 5,\allowbreak 3,\allowbreak 3,\allowbreak 2,\allowbreak 2]$, $(11,3) : [4,\allowbreak 5,\allowbreak 3,\allowbreak 2,\allowbreak 2]$

\noindent
$\rule{12.5cm}{1.1pt}$
\noindent

Example \textbf{(2.3)} shows the largest $n=29$ that we have for $K^2=2$. In relation to length, it is equal to the Wahl chain of length $7$ in Figure \ref{f2}, where $n=27$, $a=8$. 


\bigskip
\bigskip

\noindent
$\rule{12.5cm}{1.1pt}$
\noindent

\textbf{(3.1)} $K^2=3$ - $\{F_1,~\allowbreak C_1,~\allowbreak C_2,~\allowbreak C_3,~\allowbreak B_2,~\allowbreak A_1,~\allowbreak A_4,~\allowbreak D_3\}$ - $\operatorname{det}=-96$ - $~C_1 \cap C_2$, $~C_1 \cap A_4$, $~B_2 \cap D_3$, $~C_3 \cap A_1$, $[3,\allowbreak 3,\allowbreak 3,\allowbreak 1,\allowbreak 2,\allowbreak 3,\allowbreak 3] \times ~C_1 \cap A_1$ - $(34,13) : [2,\allowbreak 3,\allowbreak 3,\allowbreak 5,\allowbreak 3,\allowbreak 3,\allowbreak 3]$, $(34,13) : [2,\allowbreak 3,\allowbreak 3,\allowbreak 5,\allowbreak 3,\allowbreak 3,\allowbreak 3]$

\noindent
$\rule{12.5cm}{1.1pt}$
\noindent

Example \textbf{(3.1)} gives a T-singularity from $[2,\allowbreak 3,\allowbreak 3,\allowbreak 5,\allowbreak 3,\allowbreak 3,\allowbreak 3]-1-[2,\allowbreak 3,\allowbreak 3,\allowbreak 5,\allowbreak 3,\allowbreak 3,\allowbreak 3]$. 



\noindent
$\rule{12.5cm}{1.1pt}$
\noindent

\textbf{(3.2)} $K^2=3$ - $\{F_1,~\allowbreak C_1,~\allowbreak C_3,~\allowbreak B_2,~\allowbreak A_1,~\allowbreak A_2,~\allowbreak A_4,~\allowbreak D_3\}$ - $\operatorname{det}=-112$ - $~C_3 \cap A_2$, $~B_2 \cap A_1$, $~C_1 \cap A_2$, $[2,\allowbreak 2,1] \times ~C_1 \cap A_4$, $[2,\allowbreak 2,\allowbreak 1,\allowbreak 4] \times ~A_1 \cap C_1$ - $(59,18) : [4,\allowbreak 2,\allowbreak 2,\allowbreak 3,\allowbreak 5,\allowbreak 3,\allowbreak 5,\allowbreak 2,\allowbreak 2]$, $(7,2) : [4,\allowbreak 5,\allowbreak 2,\allowbreak 2]$

\noindent
$\rule{12.5cm}{1.1pt}$
\noindent

Example \textbf{(3.2)} gives the largest Wahl chain we know for $K^2=3$ with $\ell=9$.



\bigskip
\bigskip

\noindent
$\rule{12.5cm}{1.1pt}$
\noindent

\textbf{(4.1)} $K^2=4$ - $\{F_1,~\allowbreak C_1,~\allowbreak C_2,~\allowbreak B_1,~\allowbreak F_9,~\allowbreak A_2,~\allowbreak A_3,~\allowbreak D_1,~\allowbreak D_3,~\allowbreak D_4\}$ - $\operatorname{det}=-276$ - $~C_1 \cap C_2$, $~C_1 \cap C_2$, $~F_1 \cap A_2$, $~B_1 \cap A_3$, $[2,1] \times ~B_1 \cap A_2$, $[2,1] \times ~D_4 \cap C_2$ - $(17,7) : [3,\allowbreak 2,\allowbreak 6,\allowbreak 2,\allowbreak 4,\allowbreak 2]$, $(19,7) : [3,\allowbreak 4,\allowbreak 5,\allowbreak 2,\allowbreak 3,\allowbreak 2]$

\noindent
$\rule{12.5cm}{1.1pt}$
\noindent



\noindent
$\rule{12.5cm}{1.1pt}$
\noindent

\textbf{(4.2)} $K^2=4$ - $\{F_1,~\allowbreak F_2,~\allowbreak F_3,~\allowbreak C_1,~\allowbreak C_2,~\allowbreak B_1,~\allowbreak A_2,~\allowbreak A_3,~\allowbreak A_4,~\allowbreak D_2\}$ - $\operatorname{det}=-165$ - $~C_1 \cap C_2$, $~C_1 \cap A_2$, $~C_1 \cap A_4$, $~B_1 \cap D_2$, $[2,1] \times ~F_2 \cap F_3$, $[3,\allowbreak 2,\allowbreak 3,\allowbreak 2,\allowbreak 1,\allowbreak 3,\allowbreak 4,\allowbreak 3] \times ~C_2 \cap C_1$ - $(49,18) : [3,\allowbreak 4,\allowbreak 3,\allowbreak 5,\allowbreak 3,\allowbreak 2,\allowbreak 3,\allowbreak 2]$, $(109,40) : [3,\allowbreak 4,\allowbreak 3,\allowbreak 6,\allowbreak 2,\allowbreak 3,\allowbreak 3,\allowbreak 2,\allowbreak 3,\allowbreak 2]$

\noindent
$\rule{12.5cm}{1.1pt}$
\noindent



\noindent
$\rule{12.5cm}{1.1pt}$
\noindent

\textbf{(4.3)} $K^2=4$ - $\{F_1,~\allowbreak F_2,~\allowbreak F_3,~\allowbreak C_1,~\allowbreak C_2,~\allowbreak B_1,~\allowbreak A_2,~\allowbreak A_3,~\allowbreak A_4,~\allowbreak D_2\}$ - $\operatorname{det}=-165$ - $~C_1 \cap C_2$, $~C_1 \cap A_2$, $~C_1 \cap A_4$, $~B_1 \cap D_2$, $[2,1] \times ~F_2 \cap F_3$, $[3,\allowbreak 2,\allowbreak 3,\allowbreak 2,\allowbreak 1,\allowbreak 3,\allowbreak 4] \times ~A_3 \cap B_1$ - $(128,47) : [3,\allowbreak 4,\allowbreak 3,\allowbreak 3,\allowbreak 5,\allowbreak 3,\allowbreak 3,\allowbreak 2,\allowbreak 3,\allowbreak 2]$, $(30,11) : [3,\allowbreak 4,\allowbreak 5,\allowbreak 3,\allowbreak 2,\allowbreak 3,\allowbreak 2]$

\noindent
$\rule{12.5cm}{1.1pt}$
\noindent

Examples \textbf{(4.2)} and \textbf{(4.3)} give an example of a wormhole over the cyclic quotient singularity $\frac{1}{24964}(1,9165)$ (c.f. \cite{UV21}). The two extremal P-resolutions are:

$$[3,\allowbreak 4,\allowbreak 3,\allowbreak 5,\allowbreak 3,\allowbreak 2,\allowbreak 3,\allowbreak 2] - 1 - [3,\allowbreak 4,\allowbreak 3,\allowbreak 6,\allowbreak 2,\allowbreak 3,\allowbreak 3,\allowbreak 2,\allowbreak 3,\allowbreak 2] $$ 

$$[3,\allowbreak 4,\allowbreak 3,\allowbreak 3,\allowbreak 5,\allowbreak 3,\allowbreak 3,\allowbreak 2,\allowbreak 3,\allowbreak 2]-1-[3,\allowbreak 4,\allowbreak 5,\allowbreak 3,\allowbreak 2,\allowbreak 3,\allowbreak 2].$$



\noindent
$\rule{12.5cm}{1.1pt}$
\noindent

\textbf{(4.4)} $K^2=4$ - $\{F_1,~\allowbreak C_1,~\allowbreak C_2,~\allowbreak B_1,~\allowbreak B_4,~\allowbreak A_1,~\allowbreak A_2,~\allowbreak A_4,~\allowbreak D_1,~\allowbreak D_3\}$ - $\operatorname{det}=-240$ - $~C_1 \cap C_2$, $~C_1 \cap C_2$, $~B_4 \cap D_3$, $~F_1 \cap A_2$, $[2,1] \times ~A_4 \cap C_1$, $[2,\allowbreak 3,\allowbreak 3,\allowbreak 2,\allowbreak 1,\allowbreak 3,\allowbreak 3,\allowbreak 4,\allowbreak 2,\allowbreak 2] \times ~A_2 \cap C_1$ - $(92,35) : [3,\allowbreak 3,\allowbreak 4,\allowbreak 2,\allowbreak 2,\allowbreak 7,\allowbreak 2,\allowbreak 3,\allowbreak 3,\allowbreak 2]$, $(92,35) : [3,\allowbreak 3,\allowbreak 4,\allowbreak 2,\allowbreak 2,\allowbreak 7,\allowbreak 2,\allowbreak 3,\allowbreak 3,\allowbreak 2]$

\noindent
$\rule{12.5cm}{1.1pt}$
\noindent

Example \textbf{(4.4)} produces a T-singularity from $[3,\allowbreak 3,\allowbreak 4,\allowbreak 2,\allowbreak 2,\allowbreak 7,\allowbreak 2,\allowbreak 3,\allowbreak 3,\allowbreak 2]-1-[3,\allowbreak 3,\allowbreak 4,\allowbreak 2,\allowbreak 2,\allowbreak 7,\allowbreak 2,\allowbreak 3,\allowbreak 3,\allowbreak 2]$.



\bigskip
\bigskip

\noindent
$\rule{12.5cm}{1.1pt}$
\noindent

\textbf{(5.1)} $K^2=5$ - $\{F_1,~\allowbreak F_2,~\allowbreak F_3,~\allowbreak C_1,~\allowbreak C_2,~\allowbreak C_3,~\allowbreak B_1,~\allowbreak F_9,~\allowbreak A_2,~\allowbreak A_3,~\allowbreak A_4,~\allowbreak D_1\}$ - $\operatorname{det}=-352$ - $~C_1 \cap C_2$, $~C_1 \cap A_2$, $~C_1 \cap A_4$, $~B_1 \cap A_2$, $~F_1 \cap A_2$, $~B_1 \cap A_3$, $[2,\allowbreak 2,1] \times ~A_3 \cap C_3$ - $(27,8) : [4,\allowbreak 2,\allowbreak 3,\allowbreak 5,\allowbreak 4,\allowbreak 2,\allowbreak 2]$, $(25,11) : [3,\allowbreak 2,\allowbreak 2,\allowbreak 3,\allowbreak 5,\allowbreak 5,\allowbreak 2]$

\noindent
$\rule{12.5cm}{1.1pt}$
\noindent



\noindent
$\rule{12.5cm}{1.1pt}$
\noindent

\textbf{(5.2)} $K^2=5$ - $\{F_1,~\allowbreak F_2,~\allowbreak F_5,~\allowbreak C_1,~\allowbreak C_2,~\allowbreak B_1,~\allowbreak B_4,~\allowbreak A_1,~\allowbreak A_2,~\allowbreak A_3,~\allowbreak D_3,~\allowbreak D_4\}$ - $\operatorname{det}=-272$ - $~F_1 \cap A_2$, $~C_1 \cap C_2$, $~C_1 \cap A_1$, $~C_1 \cap A_3$, $~C_2 \cap D_4$, $~B_4 \cap D_3$, $[2,\allowbreak 3,\allowbreak 2,\allowbreak 3,\allowbreak 3,\allowbreak 3,\allowbreak 1,\allowbreak 2,\allowbreak 3,\allowbreak 3,\allowbreak 4] \times ~C_1 \cap A_2$ - $(175,67) : [2,\allowbreak 3,\allowbreak 3,\allowbreak 4,\allowbreak 6,\allowbreak 2,\allowbreak 3,\allowbreak 2,\allowbreak 3,\allowbreak 3,\allowbreak 3]$, $(175,67) : [2,\allowbreak 3,\allowbreak 3,\allowbreak 4,\allowbreak 6,\allowbreak 2,\allowbreak 3,\allowbreak 2,\allowbreak 3,\allowbreak 3,\allowbreak 3]$

\noindent
$\rule{12.5cm}{1.1pt}$
\noindent

Example \textbf{(5.2)} gives a T-singularity from $[2,\allowbreak 3,\allowbreak 3,\allowbreak 4,\allowbreak 6,\allowbreak 2,\allowbreak 3,\allowbreak 2,\allowbreak 3,\allowbreak 3,\allowbreak 3]-1-[2,\allowbreak 3,\allowbreak 3,\allowbreak 4,\allowbreak 6,\allowbreak 2,\allowbreak 3,\allowbreak 2,\allowbreak 3,\allowbreak 3,\allowbreak 3]$.



\noindent
$\rule{12.5cm}{1.1pt}$
\noindent

\textbf{(5.3)} $K^2=5$ - $\{F_1,~\allowbreak F_2,~\allowbreak F_3,~\allowbreak C_1,~\allowbreak C_2,~\allowbreak B_1,~\allowbreak B_4,~\allowbreak F_9,~\allowbreak A_2,~\allowbreak D_2,~\allowbreak D_3,~\allowbreak D_4\}$ - $\operatorname{det}=-240$ - $~F_1 \cap A_2$, $~C_1 \cap C_2$, $~C_2 \cap D_2$, $~C_2 \cap D_3$, $~B_4 \cap D_4$, $[2,1] \times ~F_9 \cap A_2$, $[5,\allowbreak 3,\allowbreak 3,\allowbreak 2,\allowbreak 1,\allowbreak 3,\allowbreak 3,\allowbreak 3,\allowbreak 2,\allowbreak 2] \times ~A_2 \cap C_1$ - $(227,87) : [3,\allowbreak 3,\allowbreak 3,\allowbreak 2,\allowbreak 2,\allowbreak 3,\allowbreak 2,\allowbreak 6,\allowbreak 5,\allowbreak 3,\allowbreak 3,\allowbreak 2]$, $(107,41) : [3,\allowbreak 3,\allowbreak 3,\allowbreak 2,\allowbreak 2,\allowbreak 5,\allowbreak 5,\allowbreak 3,\allowbreak 3,\allowbreak 2]$

\noindent
$\rule{12.5cm}{1.1pt}$
\noindent



\noindent
$\rule{12.5cm}{1.1pt}$
\noindent

\textbf{(5.4)} $K^2=5$ - $\{F_1,~\allowbreak F_2,~\allowbreak F_3,~\allowbreak C_1,~\allowbreak C_2,~\allowbreak B_1,~\allowbreak B_4,~\allowbreak F_9,~\allowbreak A_2,~\allowbreak D_2,~\allowbreak D_3,~\allowbreak D_4\}$ - $\operatorname{det}=-240$ - $~F_1 \cap A_2$, $~C_1 \cap C_2$, $~C_2 \cap D_2$, $~C_2 \cap D_3$, $~B_4 \cap D_4$, $[2,1] \times ~F_9 \cap A_2$, $[3,\allowbreak 3,\allowbreak 2,\allowbreak 1,\allowbreak 3,\allowbreak 3,\allowbreak 3,\allowbreak 2,\allowbreak 2] \times ~D_2 \cap B_1$ - $(60,23) : [3,\allowbreak 3,\allowbreak 3,\allowbreak 2,\allowbreak 2,\allowbreak 7,\allowbreak 3,\allowbreak 3,\allowbreak 2]$, $(274,105) : [3,\allowbreak 3,\allowbreak 3,\allowbreak 2,\allowbreak 2,\allowbreak 3,\allowbreak 5,\allowbreak 3,\allowbreak 5,\allowbreak 3,\allowbreak 3,\allowbreak 2]$

\noindent
$\rule{12.5cm}{1.1pt}$
\noindent

Examples \textbf{(5.3)} and \textbf{(5.4)} give an example of a wormhole over the cyclic quotient singularity $\frac{1}{111556}(1,42753)$ (c.f. \cite{UV21}). The two extremal P-resolutions are: $$[3,\allowbreak 3,\allowbreak 3,\allowbreak 2,\allowbreak 2,\allowbreak 3,\allowbreak 2,\allowbreak 6,\allowbreak 5,\allowbreak 3,\allowbreak 3,\allowbreak 2]-1-[3,\allowbreak 3,\allowbreak 3,\allowbreak 2,\allowbreak 2,\allowbreak 5,\allowbreak 5,\allowbreak 3,\allowbreak 3,\allowbreak 2]$$ $$[3,\allowbreak 3,\allowbreak 3,\allowbreak 2,\allowbreak 2,\allowbreak 7,\allowbreak 3,\allowbreak 3,\allowbreak 2]-1-[3,\allowbreak 3,\allowbreak 3,\allowbreak 2,\allowbreak 2,\allowbreak 3,\allowbreak 5,\allowbreak 3,\allowbreak 5,\allowbreak 3,\allowbreak 3,\allowbreak 2]$$ We recall that the longest Wahl chain we know for $K^2=5$ is in Figure \ref{f5} where $\ell=14$.



\bigskip
\bigskip

\noindent
$\rule{12.5cm}{1.1pt}$
\noindent

\textbf{(6.1)} $K^2=6$ - $\{F_1,~\allowbreak C_1,~\allowbreak C_2,~\allowbreak C_3,~\allowbreak B_1,~\allowbreak B_3,~\allowbreak F_9,~\allowbreak F_{14},~\allowbreak F_{15},~\allowbreak F_{16},~\allowbreak A_1,~\allowbreak A_2,~\allowbreak A_3,~\allowbreak D_3\}$ - $\operatorname{det}=-240$ - $~C_1 \cap C_2$, $~C_1 \cap A_2$, $~C_1 \cap A_3$, $~F_{15} \cap A_1$, $~F_1 \cap A_2$, $~C_3 \cap A_2$, $~B_3 \cap A_2$, $[2,1] \times ~C_3 \cap A_3$ - $(31,13) : [3,\allowbreak 2,\allowbreak 3,\allowbreak 5,\allowbreak 3,\allowbreak 4,\allowbreak 2]$, $(36,13) : [3,\allowbreak 5,\allowbreak 2,\allowbreak 6,\allowbreak 2,\allowbreak 2,\allowbreak 3,\allowbreak 2]$

\noindent
$\rule{12.5cm}{1.1pt}$
\noindent



\noindent
$\rule{12.5cm}{1.1pt}$
\noindent

\textbf{(6.2)} $K^2=6$ - $\{F_1,~\allowbreak C_1,~\allowbreak C_2,~\allowbreak C_3,~\allowbreak B_1,~\allowbreak F_{11},~\allowbreak F_{12},~\allowbreak F_{13},~\allowbreak F_{14},~\allowbreak F_{15},~\allowbreak A_1,~\allowbreak A_2,~\allowbreak A_3,~\allowbreak D_1\}$ - $\operatorname{det}=-240$ - $~C_1 \cap C_2$, $~C_1 \cap A_1$, $~C_1 \cap A_3$, $~C_3 \cap A_2$, $~B_1 \cap A_2$, $~F_{13} \cap A_3$, $~F_{11} \cap D_1$, $[3,\allowbreak 3,\allowbreak 2,\allowbreak 2,\allowbreak 3,\allowbreak 2,\allowbreak 3,\allowbreak 1,\allowbreak 2,\allowbreak 4,\allowbreak 5,\allowbreak 3,\allowbreak 3] \times ~B_1 \cap A_3$ - $(394,167) : [2,\allowbreak 4,\allowbreak 5,\allowbreak 3,\allowbreak 3,\allowbreak 5,\allowbreak 3,\allowbreak 3,\allowbreak 2,\allowbreak 2,\allowbreak 3,\allowbreak 2,\allowbreak 3]$, $(394,167) : [2,\allowbreak 4,\allowbreak 5,\allowbreak 3,\allowbreak 3,\allowbreak 5,\allowbreak 3,\allowbreak 3,\allowbreak 2,\allowbreak 2,\allowbreak 3,\allowbreak 2,\allowbreak 3]$

\noindent
$\rule{12.5cm}{1.1pt}$
\noindent

Example \textbf{(6.2)} defines the T-singularity $\frac{1}{2 \cdot 394^2 }(1,2 \cdot 394 \cdot 167-1)$.


\noindent
$\rule{12.5cm}{1.1pt}$
\noindent

\textbf{(6.3)} $K^2=6$ - $\{F_1,~\allowbreak F_2,~\allowbreak F_3,~\allowbreak F_4,~\allowbreak F_7,~\allowbreak F_8,~\allowbreak C_1,~\allowbreak C_2,~\allowbreak C_3,~\allowbreak B_1,~\allowbreak A_1,~\allowbreak A_2,~\allowbreak A_3,~\allowbreak D_1\}$ - $\operatorname{det}=-272$ - $~F_1 \cap A_2$, $~F_7 \cap A_1$, $~C_1 \cap C_2$, $~C_1 \cap A_2$, $~C_1 \cap A_3$, $~C_3 \cap A_2$, $~B_1 \cap D_1$, $[5,\allowbreak 1,\allowbreak 2,\allowbreak 2,\allowbreak 2] \times ~A_3 \cap B_1$ - $(359,76) : [2,\allowbreak 2,\allowbreak 2,\allowbreak 3,\allowbreak 2,\allowbreak 3,\allowbreak 3,\allowbreak 3,\allowbreak 5,\allowbreak 3,\allowbreak 3,\allowbreak 4,\allowbreak 5]$, $(9,2) : [2,\allowbreak 2,\allowbreak 2,\allowbreak 5,\allowbreak 5]$

\noindent
$\rule{12.5cm}{1.1pt}$
\noindent



\noindent
$\rule{12.5cm}{1.1pt}$
\noindent

\textbf{(6.4)} $K^2=6$ - $\{F_1,~\allowbreak F_2,~\allowbreak F_3,~\allowbreak F_4,~\allowbreak F_7,~\allowbreak F_8,~\allowbreak C_1,~\allowbreak C_2,~\allowbreak C_3,~\allowbreak B_1,~\allowbreak A_1,~\allowbreak A_2,~\allowbreak A_3,~\allowbreak D_1\}$ - $\operatorname{det}=-272$ - $~F_1 \cap A_2$, $~F_7 \cap A_1$, $~C_1 \cap C_2$, $~C_1 \cap A_2$, $~C_1 \cap A_3$, $~C_3 \cap A_2$, $~B_1 \cap D_1$, $[1,\allowbreak 2,\allowbreak 2,\allowbreak 2] \times ~F_1 \cap F_8$ - $(5,1) : [2,\allowbreak 2,\allowbreak 2,\allowbreak 7]$, $(411,89) : [2,\allowbreak 2,\allowbreak 2,\allowbreak 3,\allowbreak 3,\allowbreak 3,\allowbreak 3,\allowbreak 5,\allowbreak 3,\allowbreak 3,\allowbreak 3,\allowbreak 3,\allowbreak 5]$

\noindent
$\rule{12.5cm}{1.1pt}$
\noindent

Examples \textbf{(6.3)} and \textbf{(6.4)} give an example of a wormhole over the cyclic quotient singularity $\frac{1}{238816}(1,188257)$ (c.f. \cite{UV21}). The two extremal P-resolutions are:

$$[2,\allowbreak 2,\allowbreak 2,\allowbreak 3,\allowbreak 2,\allowbreak 3,\allowbreak 3,\allowbreak 3,\allowbreak 5,\allowbreak 3,\allowbreak 3,\allowbreak 4,\allowbreak 5]-1-[2,\allowbreak 2,\allowbreak 2,\allowbreak 5,\allowbreak 5]$$

$$[2,\allowbreak 2,\allowbreak 2,\allowbreak 7]-1-[2,\allowbreak 2,\allowbreak 2,\allowbreak 3,\allowbreak 3,\allowbreak 3,\allowbreak 3,\allowbreak 5,\allowbreak 3,\allowbreak 3,\allowbreak 3,\allowbreak 3,\allowbreak 5]$$



\bigskip
\bigskip

\noindent
$\rule{12.5cm}{1.1pt}$
\noindent

\textbf{(7.1)} $K^2=7$ - $\{F_1,~\allowbreak F_2,~\allowbreak F_3,~\allowbreak F_5,~\allowbreak F_6,~\allowbreak C_1,~\allowbreak C_2,~\allowbreak B_1,~\allowbreak B_4,~\allowbreak F_{13},~\allowbreak A_1,~\allowbreak A_2,~\allowbreak A_3,~\allowbreak A_4,~\allowbreak D_1,~\allowbreak D_3\}$ - $\operatorname{det}=-336$ - $~C_1 \cap C_2$, $~C_1 \cap C_2$, $~C_1 \cap A_3$, $~C_1 \cap A_4$, $~B_1 \cap A_3$, $~B_4 \cap D_3$, $~B_1 \cap A_2$, $~F_1 \cap F_2$, $[2,1] \times ~F_1 \cap D_3$ - $(69,29) : [3,\allowbreak 2,\allowbreak 3,\allowbreak 3,\allowbreak 2,\allowbreak 6,\allowbreak 3,\allowbreak 4,\allowbreak 2]$, $(41,12) : [2,\allowbreak 2,\allowbreak 4,\allowbreak 2,\allowbreak 5,\allowbreak 4,\allowbreak 2,\allowbreak 4]$

\noindent
$\rule{12.5cm}{1.1pt}$
\noindent



\bigskip
\bigskip

\noindent
$\rule{12.5cm}{1.1pt}$
\noindent

\textbf{(7.2)} $K^2=7$ - $\{F_1,~\allowbreak F_2,~\allowbreak F_3,~\allowbreak F_4,~\allowbreak F_5,~\allowbreak F_6,~\allowbreak F_7,~\allowbreak C_1,~\allowbreak C_2,~\allowbreak B_1,~\allowbreak B_4,~\allowbreak A_1,~\allowbreak A_2,~\allowbreak A_3,~\allowbreak D_1,~\allowbreak D_3\}$ - $\operatorname{det}=-64$ - $~F_5 \cap A_3$, $~F_7 \cap D_1$, $~C_1 \cap C_2$, $~C_1 \cap A_1$, $~C_1 \cap A_2$, $~C_2 \cap D_3$, $~B_1 \cap A_3$, $[2,\allowbreak 2,1] \times ~F_2 \cap F_1$, $[3,\allowbreak 2,\allowbreak 3,\allowbreak 2,\allowbreak 2,\allowbreak 3,\allowbreak 2,\allowbreak 2,\allowbreak 1,\allowbreak 4,\allowbreak 5,\allowbreak 4,\allowbreak 3] \times ~B_1 \cap A_2$ - $(322,85) : [4,\allowbreak 5,\allowbreak 4,\allowbreak 3,\allowbreak 5,\allowbreak 3,\allowbreak 2,\allowbreak 3,\allowbreak 2,\allowbreak 2,\allowbreak 3,\allowbreak 2,\allowbreak 2]$, $(1951,515) : [4,\allowbreak 5,\allowbreak 4,\allowbreak 3,\allowbreak 4,\allowbreak 5,\allowbreak 3,\allowbreak 2,\allowbreak 3,\allowbreak 3,\allowbreak 2,\allowbreak 3,\allowbreak 2,\allowbreak 2,\allowbreak 3,\allowbreak 2,\allowbreak 2]$

\noindent
$\rule{12.5cm}{1.1pt}$
\noindent



\noindent
$\rule{12.5cm}{1.1pt}$
\noindent

\textbf{(7.3)} $K^2=7$ - $\{F_1,~\allowbreak F_2,~\allowbreak F_3,~\allowbreak F_4,~\allowbreak F_5,~\allowbreak F_6,~\allowbreak F_7,~\allowbreak C_1,~\allowbreak C_2,~\allowbreak B_1,~\allowbreak B_4,~\allowbreak A_1,~\allowbreak A_2,~\allowbreak A_3,~\allowbreak D_1,~\allowbreak D_3\}$ - $\operatorname{det}=-64$ - $~F_5 \cap A_3$, $~F_7 \cap D_1$, $~C_1 \cap C_2$, $~C_1 \cap A_1$, $~C_1 \cap A_2$, $~C_2 \cap D_3$, $~B_1 \cap A_3$, $[2,\allowbreak 2,1] \times ~F_2 \cap F_1$, $[3,\allowbreak 2,\allowbreak 3,\allowbreak 2,\allowbreak 2,\allowbreak 3,\allowbreak 2,\allowbreak 2,\allowbreak 1,\allowbreak 4,\allowbreak 5,\allowbreak 4] \times ~B_4 \cap A_1$ - $(2201,581) : [4,\allowbreak 5,\allowbreak 4,\allowbreak 3,\allowbreak 3,\allowbreak 3,\allowbreak 5,\allowbreak 3,\allowbreak 3,\allowbreak 3,\allowbreak 2,\allowbreak 3,\allowbreak 2,\allowbreak 2,\allowbreak 3,\allowbreak 2,\allowbreak 2]$, $(197,52) : [4,\allowbreak 5,\allowbreak 4,\allowbreak 5,\allowbreak 3,\allowbreak 2,\allowbreak 3,\allowbreak 2,\allowbreak 2,\allowbreak 3,\allowbreak 2,\allowbreak 2]$

\noindent
$\rule{12.5cm}{1.1pt}$
\noindent

Examples \textbf{(7.2)} and \textbf{(7.3)} give an example of a wormhole over the cyclic quotient singularity $\frac{1}{7051195}(1,1861309)$ (c.f. \cite{UV21}). The two extremal P-resolutions are:

$$[4,\allowbreak 5,\allowbreak 4,\allowbreak 3,\allowbreak 5,\allowbreak 3,\allowbreak 2,\allowbreak 3,\allowbreak 2,\allowbreak 2,\allowbreak 3,\allowbreak 2,\allowbreak 2]-1-[4,\allowbreak 5,\allowbreak 4,\allowbreak 3,\allowbreak 4,\allowbreak 5,\allowbreak 3,\allowbreak 2,\allowbreak 3,\allowbreak 3,\allowbreak 2,\allowbreak 3,\allowbreak 2,\allowbreak 2,\allowbreak 3,\allowbreak 2,\allowbreak 2]$$

$$[4,\allowbreak 5,\allowbreak 4,\allowbreak 3,\allowbreak 3,\allowbreak 3,\allowbreak 5,\allowbreak 3,\allowbreak 3,\allowbreak 3,\allowbreak 2,\allowbreak 3,\allowbreak 2,\allowbreak 2,\allowbreak 3,\allowbreak 2,\allowbreak 2]-1-[4,\allowbreak 5,\allowbreak 4,\allowbreak 5,\allowbreak 3,\allowbreak 2,\allowbreak 3,\allowbreak 2,\allowbreak 2,\allowbreak 3,\allowbreak 2,\allowbreak 2]$$



\bigskip
\bigskip

\noindent
$\rule{12.5cm}{1.1pt}$
\noindent

\textbf{(8.1)} $K^2=8$ - $\{F_1,~\allowbreak F_2,~\allowbreak F_3,~\allowbreak F_5,~\allowbreak F_6,~\allowbreak F_7,~\allowbreak C_1,~\allowbreak C_2,~\allowbreak B_1,~\allowbreak B_4,~\allowbreak F_9,~\allowbreak F_{10},~\allowbreak A_1,~\allowbreak A_2,~\allowbreak A_4,~\allowbreak D_1,~\allowbreak D_2,~\allowbreak D_3\}$ - $\operatorname{det}=-256$ - $~F_1 \cap A_2$, $~F_3 \cap A_4$, $~F_7 \cap D_1$, $~C_1 \cap C_2$, $~C_1 \cap C_2$, $~C_1 \cap A_2$, $~C_2 \cap D_3$, $~B_1 \cap D_2$, $~B_4 \cap A_1$, $~D_3 \cap F_1$ - $(47,13) : [2,\allowbreak 2,\allowbreak 3,\allowbreak 3,\allowbreak 5,\allowbreak 3,\allowbreak 3,\allowbreak 4]$, $(115,34) : [4,\allowbreak 2,\allowbreak 3,\allowbreak 3,\allowbreak 5,\allowbreak 3,\allowbreak 3,\allowbreak 4,\allowbreak 2,\allowbreak 2]$

\noindent
$\rule{12.5cm}{1.1pt}$
\noindent



\noindent
$\rule{12.5cm}{1.1pt}$
\noindent

\textbf{(8.2)} $K^2=8$ - $\{F_4,~\allowbreak F_5,~\allowbreak F_6,~\allowbreak F_7,~\allowbreak C_1,~\allowbreak C_2,~\allowbreak C_4,~\allowbreak B_1,~\allowbreak F_9,~\allowbreak F_{10},~\allowbreak F_{11},~\allowbreak F_{15},~\allowbreak F_{16},~\allowbreak A_1,~\allowbreak A_2,~\allowbreak D_1,~\allowbreak D_2,~\allowbreak D_3\}$ - $\operatorname{det}=-208$ - $~F_7 \cap D_1$, $~C_1 \cap C_2$, $~C_1 \cap A_1$, $~C_2 \cap D_1$, $~C_2 \cap D_2$, $~C_4 \cap D_2$, $~B_1 \cap D_1$, $~F_9 \cap A_2$, $~F_{15} \cap D_2$, $[2,\allowbreak 5,\allowbreak 4,\allowbreak 3,\allowbreak 3,\allowbreak 5,\allowbreak 1,\allowbreak 2,\allowbreak 2,\allowbreak 2,\allowbreak 3,\allowbreak 3,\allowbreak 3,\allowbreak 2,\allowbreak 3,\allowbreak 2,\allowbreak 2] \times ~D_1 \cap C_4$ - $(1631,353) : [2,\allowbreak 2,\allowbreak 2,\allowbreak 3,\allowbreak 3,\allowbreak 3,\allowbreak 2,\allowbreak 3,\allowbreak 2,\allowbreak 2,\allowbreak 6,\allowbreak 2,\allowbreak 5,\allowbreak 4,\allowbreak 3,\allowbreak 3,\allowbreak 5]$, $(1631,353) : [2,\allowbreak 2,\allowbreak 2,\allowbreak 3,\allowbreak 3,\allowbreak 3,\allowbreak 2,\allowbreak 3,\allowbreak 2,\allowbreak 2,\allowbreak 6,\allowbreak 2,\allowbreak 5,\allowbreak 4,\allowbreak 3,\allowbreak 3,\allowbreak 5]$

\noindent
$\rule{12.5cm}{1.1pt}$
\noindent

Example \textbf{(8.2)} defines the T-singularity $\frac{1}{2 \cdot 1631^2 }(1,2 \cdot 1631 \cdot 353-1)$.



\noindent
$\rule{12.5cm}{1.1pt}$
\noindent

\textbf{(8.3)} $K^2=8$ - $\{F_1,~\allowbreak F_2,~\allowbreak F_3,~\allowbreak F_7,~\allowbreak F_8,~\allowbreak C_1,~\allowbreak C_2,~\allowbreak C_4,~\allowbreak B_1,~\allowbreak F_{11},~\allowbreak F_{12},~\allowbreak F_{13},~\allowbreak F_{14},~\allowbreak F_{15},~\allowbreak A_1,~\allowbreak A_2,~\allowbreak D_1,~\allowbreak D_2\}$ - $\operatorname{det}=-160$ - $~F_1 \cap A_2$, $~F_7 \cap D_1$, $~C_1 \cap C_2$, $~C_1 \cap A_1$, $~C_2 \cap D_1$, $~B_1 \cap D_2$, $~F_{15} \cap D_2$, $~C_4 \cap D_2$, $[2,1] \times ~F_{11} \cap D_1$, $[3,\allowbreak 2,\allowbreak 3,\allowbreak 3,\allowbreak 3,\allowbreak 2,\allowbreak 2,\allowbreak 2,\allowbreak 3,\allowbreak 2,\allowbreak 1,\allowbreak 3,\allowbreak 6,\allowbreak 3,\allowbreak 3,\allowbreak 4] \times ~C_2 \cap D_2$ - $(1253,444) : [3,\allowbreak 6,\allowbreak 3,\allowbreak 3,\allowbreak 4,\allowbreak 5,\allowbreak 3,\allowbreak 2,\allowbreak 3,\allowbreak 3,\allowbreak 3,\allowbreak 2,\allowbreak 2,\allowbreak 2,\allowbreak 3,\allowbreak 2]$, $(2173,770) : [3,\allowbreak 6,\allowbreak 3,\allowbreak 3,\allowbreak 4,\allowbreak 7,\allowbreak 2,\allowbreak 2,\allowbreak 3,\allowbreak 2,\allowbreak 3,\allowbreak 3,\allowbreak 3,\allowbreak 2,\allowbreak 2,\allowbreak 2,\allowbreak 3,\allowbreak 2]$

\noindent
$\rule{12.5cm}{1.1pt}$
\noindent



\noindent
$\rule{12.5cm}{1.1pt}$
\noindent

\textbf{(8.4)} $K^2=8$ - $\{F_1,~\allowbreak F_2,~\allowbreak F_3,~\allowbreak F_7,~\allowbreak F_8,~\allowbreak C_1,~\allowbreak C_2,~\allowbreak C_4,~\allowbreak B_1,~\allowbreak F_{11},~\allowbreak F_{12},~\allowbreak F_{13},~\allowbreak F_{14},~\allowbreak F_{15},~\allowbreak A_1,~\allowbreak A_2,~\allowbreak D_1,~\allowbreak D_2\}$ - $\operatorname{det}=-160$ - $~F_1 \cap A_2$, $~F_7 \cap D_1$, $~C_1 \cap C_2$, $~C_1 \cap A_1$, $~C_2 \cap D_1$, $~B_1 \cap D_2$, $~F_{15} \cap D_2$, $~C_4 \cap D_2$, $[2,1] \times ~F_{11} \cap D_1$, $[2,\allowbreak 3,\allowbreak 3,\allowbreak 3,\allowbreak 2,\allowbreak 2,\allowbreak 2,\allowbreak 3,\allowbreak 2,\allowbreak 1,\allowbreak 3,\allowbreak 6,\allowbreak 3,\allowbreak 3] \times ~F_2 \cap F_1$ - $(2966,1051) : [3,\allowbreak 6,\allowbreak 3,\allowbreak 3,\allowbreak 4,\allowbreak 4,\allowbreak 5,\allowbreak 2,\allowbreak 3,\allowbreak 2,\allowbreak 3,\allowbreak 3,\allowbreak 3,\allowbreak 2,\allowbreak 2,\allowbreak 2,\allowbreak 3,\allowbreak 2]$, $(460,163) : [3,\allowbreak 6,\allowbreak 3,\allowbreak 3,\allowbreak 6,\allowbreak 2,\allowbreak 3,\allowbreak 3,\allowbreak 3,\allowbreak 2,\allowbreak 2,\allowbreak 2,\allowbreak 3,\allowbreak 2]$

\noindent
$\rule{12.5cm}{1.1pt}$
\noindent

Examples \textbf{(8.3)} and \textbf{(8.4)} give an example of a wormhole over the cyclic quotient singularity $\frac{1}{11737476}(1,4159165)$ (c.f. \cite{UV21}). The two extremal P-resolutions are:
\bigskip

$[3,\allowbreak 6,\allowbreak 3,\allowbreak 3,\allowbreak 4,\allowbreak 5,\allowbreak 3,\allowbreak 2,\allowbreak 3,\allowbreak 3,\allowbreak 3,\allowbreak 2,\allowbreak 2,\allowbreak 2,\allowbreak 3,\allowbreak 2]-1-[3,\allowbreak 6,\allowbreak 3,\allowbreak 3,\allowbreak 4,\allowbreak 7,\allowbreak 2,\allowbreak 2,\allowbreak 3,\allowbreak 2,\allowbreak 3,\allowbreak 3,\allowbreak 3,\allowbreak 2,\allowbreak 2,\allowbreak 2,\allowbreak 3,\allowbreak 2]$

$[3,\allowbreak 6,\allowbreak 3,\allowbreak 3,\allowbreak 4,\allowbreak 4,\allowbreak 5,\allowbreak 2,\allowbreak 3,\allowbreak 2,\allowbreak 3,\allowbreak 3,\allowbreak 3,\allowbreak 2,\allowbreak 2,\allowbreak 2,\allowbreak 3,\allowbreak 2]-1-[3,\allowbreak 6,\allowbreak 3,\allowbreak 3,\allowbreak 6,\allowbreak 2,\allowbreak 3,\allowbreak 3,\allowbreak 3,\allowbreak 2,\allowbreak 2,\allowbreak 2,\allowbreak 3,\allowbreak 2]$


\bigskip
\bigskip

\noindent
$\rule{12.5cm}{1.1pt}$
\noindent

\textbf{(9.1)} $K^2=9$ - $\{F_1,~\allowbreak F_2,~\allowbreak F_3,~\allowbreak F_4,~\allowbreak F_5,~\allowbreak F_6,~\allowbreak F_7,~\allowbreak C_1,~\allowbreak C_2,~\allowbreak B_1,~\allowbreak F_9,~\allowbreak F_{11},~\allowbreak F_{12},~\allowbreak F_{15},~\allowbreak F_{16},~\allowbreak A_1,~\allowbreak A_2,~\allowbreak A_3,~\allowbreak D_1,~\allowbreak D_2\}$ - $\operatorname{det}=-64$ - $~F_3 \cap D_2$, $~F_5 \cap A_3$, $~F_7 \cap D_1$, $~C_1 \cap C_2$, $~C_1 \cap C_2$, $~C_1 \cap A_2$, $~B_1 \cap A_2$, $~B_1 \cap D_1$, $~F_{15} \cap D_2$, $~F_{15} \cap A_1$, $[2,\allowbreak 4,\allowbreak 1,\allowbreak 2,\allowbreak 2] \times ~F_1 \cap A_2$ - $(3706,1101) : [2,\allowbreak 2,\allowbreak 4,\allowbreak 4,\allowbreak 4,\allowbreak 4,\allowbreak 3,\allowbreak 5,\allowbreak 3,\allowbreak 3,\allowbreak 2,\allowbreak 3,\allowbreak 2,\allowbreak 3,\allowbreak 2,\allowbreak 3,\allowbreak 2,\allowbreak 4]$, $(13,4) : [2,\allowbreak 2,\allowbreak 7,\allowbreak 2,\allowbreak 2,\allowbreak 4]$

\noindent
$\rule{12.5cm}{1.1pt}$
\noindent



\noindent
$\rule{12.5cm}{1.1pt}$
\noindent

\textbf{(9.2)} $K^2=9$ - $\{F_1,~\allowbreak F_2,~\allowbreak F_3,~\allowbreak F_4,~\allowbreak F_5,~\allowbreak F_6,~\allowbreak F_7,~\allowbreak C_1,~\allowbreak C_2,~\allowbreak B_1,~\allowbreak F_9,~\allowbreak F_{11},~\allowbreak F_{12},~\allowbreak F_{15},~\allowbreak F_{16},~\allowbreak A_1,~\allowbreak A_2,~\allowbreak A_3,~\allowbreak D_1,~\allowbreak D_2\}$ - $\operatorname{det}=-64$ - $~F_3 \cap D_2$, $~F_5 \cap A_3$, $~F_7 \cap D_1$, $~C_1 \cap C_2$, $~C_1 \cap C_2$, $~C_1 \cap A_2$, $~B_1 \cap A_2$, $~B_1 \cap D_1$, $~F_{15} \cap D_2$, $~F_7 \cap A_1$, $[4,\allowbreak 4,\allowbreak 2,\allowbreak 2,\allowbreak 1,\allowbreak 4,\allowbreak 2,\allowbreak 3,\allowbreak 2,\allowbreak 3,\allowbreak 2] \times ~A_2 \cap F_9$ - $(175,52) : [4,\allowbreak 2,\allowbreak 3,\allowbreak 2,\allowbreak 3,\allowbreak 2,\allowbreak 2,\allowbreak 7,\allowbreak 4,\allowbreak 4,\allowbreak 2,\allowbreak 2]$, $(3706,1101) : [4,\allowbreak 2,\allowbreak 3,\allowbreak 2,\allowbreak 3,\allowbreak 2,\allowbreak 3,\allowbreak 2,\allowbreak 3,\allowbreak 3,\allowbreak 5,\allowbreak 3,\allowbreak 4,\allowbreak 4,\allowbreak 4,\allowbreak 4,\allowbreak 2,\allowbreak 2]$

\noindent
$\rule{12.5cm}{1.1pt}$
\noindent


Examples \textbf{(9.1)} and \textbf{(9.2)} come from the same configuration, but they are not wormholes. Moreover, one of the Wahl singularities is the same.

\noindent
$\rule{12.5cm}{1.1pt}$
\noindent

\textbf{(9.3)} $K^2=9$ - $\{F_1,~\allowbreak F_2,~\allowbreak F_3,~\allowbreak F_4,~\allowbreak F_5,~\allowbreak F_7,~\allowbreak F_8,~\allowbreak C_1,~\allowbreak C_2,~\allowbreak B_2,~\allowbreak F_9,~\allowbreak F_{13},~\allowbreak F_{15},~\allowbreak F_{16},~\allowbreak A_1,~\allowbreak A_2,~\allowbreak A_3,~\allowbreak D_1,~\allowbreak D_2,~\allowbreak D_3\}$ - $\operatorname{det}=-16$ - $~F_1 \cap A_2$, $~F_1 \cap D_3$, $~F_3 \cap F_4$, $~F_7 \cap A_1$, $~C_1 \cap C_2$, $~C_1 \cap A_1$, $~C_1 \cap A_3$, $~C_2 \cap D_2$, $~C_2 \cap D_3$, $[2,1] \times ~D_2 \cap F_{15}$, $[2,\allowbreak 1,\allowbreak 3,\allowbreak 2] \times ~A_3 \cap F_{13}$ - $(7,3) : [3,\allowbreak 2,\allowbreak 6,\allowbreak 2]$, $(8294,3475) : [3,\allowbreak 2,\allowbreak 3,\allowbreak 4,\allowbreak 2,\allowbreak 4,\allowbreak 4,\allowbreak 2,\allowbreak 2,\allowbreak 3,\allowbreak 5,\allowbreak 5,\allowbreak 2,\allowbreak 3,\allowbreak 2,\allowbreak 4,\allowbreak 2,\allowbreak 3,\allowbreak 4,\allowbreak 2]$

\noindent
$\rule{12.5cm}{1.1pt}$
\noindent


\noindent
$\rule{12.5cm}{1.1pt}$
\noindent

\textbf{(9.4)} $K^2=9$ - $\{F_1,~\allowbreak F_2,~\allowbreak F_3,~\allowbreak F_4,~\allowbreak F_5,~\allowbreak F_6,~\allowbreak F_7,~\allowbreak C_1,~\allowbreak C_2,~\allowbreak B_1,~\allowbreak F_9,~\allowbreak F_{10},~\allowbreak F_{11},~\allowbreak F_{12},~\allowbreak F_{13},~\allowbreak A_1,~\allowbreak A_2,~\allowbreak A_3,~\allowbreak D_1,~\allowbreak D_2\}$ - $\operatorname{det}=-16$ - $~F_3 \cap D_2$, $~F_5 \cap A_3$, $~F_7 \cap D_1$, $~C_1 \cap C_2$, $~C_1 \cap C_2$, $~C_1 \cap A_3$, $~B_1 \cap A_2$, $~B_1 \cap D_2$, $~F_{11} \cap D_1$, $[2,\allowbreak 2,1] \times ~F_1 \cap A_2$, $[2,\allowbreak 3,\allowbreak 2,\allowbreak 3,\allowbreak 2,\allowbreak 3,\allowbreak 2,\allowbreak 3,\allowbreak 2,\allowbreak 2,\allowbreak 1,\allowbreak 4,\allowbreak 4,\allowbreak 4,\allowbreak 4,\allowbreak 4,\allowbreak 2,\allowbreak 2] \times ~F_{11} \cap F_{10}$ - $(2493,668) : [4,\allowbreak 4,\allowbreak 4,\allowbreak 4,\allowbreak 4,\allowbreak 2,\allowbreak 2,\allowbreak 7,\allowbreak 2,\allowbreak 3,\allowbreak 2,\allowbreak 3,\allowbreak 2,\allowbreak 3,\allowbreak 2,\allowbreak 3,\allowbreak 2,\allowbreak 2]$, $(9401,2519) : [4,\allowbreak 4,\allowbreak 4,\allowbreak 4,\allowbreak 4,\allowbreak 2,\allowbreak 2,\allowbreak 3,\allowbreak 2,\allowbreak 6,\allowbreak 5,\allowbreak 2,\allowbreak 3,\allowbreak 2,\allowbreak 3,\allowbreak 2,\allowbreak 3,\allowbreak 2,\allowbreak 3,\allowbreak 2,\allowbreak 2]$

\noindent
$\rule{12.5cm}{1.1pt}$
\noindent

Example \textbf{(9.4)} shows the largest sum of lengths of Wahl chains we have found. 





\begin{thebibliography}{99}


\bibitem[BHPV04]{BHPV04}
    W. P. Barth, K. Hulek, C. A. M. Peters, A. Van de Ven,
    \emph{Compact complex surfaces},
    Ergebnisse der Mathematik und ihrer Grenzgebiete. 3. Folge., second edition, vol. 4, Springer-Verlag, Berlin, 2004.

\bibitem[B14]{B14}
    A. Beauville,
    \emph{Some surfaces with maximal Picard number},
    J. \'Ec. polytech. Math. 1 (2014), 101--116.
    
\bibitem[C79]{C79}
    F. Catanese,
    \emph{Surfaces with $K^2=p_g=1$ and their period mapping}, 
    Algebraic geometry (Proc. Summer Meeting, Univ. Copenhagen, Copenhagen, 1978), pp. 1--29, Lecture Notes in Math., 732, Springer, Berlin, 1979.

\bibitem[C80]{C80}
    F. Catanese,
    \emph{The moduli and the global period mapping of surfaces with $K^2=p_g=1$: a counterexample to the global Torelli problem},
    Compositio Math. 41(1980), no.3, 401--414.

\bibitem[CD80]{CD80}
    F. Catanese, O. Debarre,
    \emph{Surfaces with $K^2=2$, $p_g=1$, $q=0$},
    J. Reine Angew. Math. 395(1989), 1--55.
    
\bibitem[CCO94]{CCO94}
    F. Catanese, P. Cragnolini, P. Oliverio,
    \emph{Surfaces with $K^2=\chi=2$ and special nets of quadrics in 3-space}, Classification of algebraic varieties (L'Aquila, 1992), 77--128, Contemp. Math., 162, Amer. Math. Soc., Providence, RI, 1994.

\bibitem[Co16]{Co16}    
    S. Coughlan, 
    \emph{Hyperelliptic K3 surfaces, and Godeaux surfaces with $\pi_1=\Z/2$}, 
    J. Korean Math. Soc. 53 (2016), no. 4, 869--893.

\bibitem[CU18]{CU18}
	S. Coughlan, G. Urz\'ua,
	\emph{On $\mathbb{Z}$/3-Godeaux surfaces},
	Int. Math. Res. Not. IMRN 2018, no. 18, 5609--5637.

\bibitem[DR20]{DR20}
	E. Dias, C. Rito,
	\emph{$\Z/2$-Godeaux surfaces},
	pre-print arXiv:2009.12645 [math.AG].

\bibitem[DRU20]{DRU20}
	E. Dias, C. Rito, G. Urz\'ua,
	\emph{On degenerations of $\Z/2$-Godeaux surfaces},
	pre-print arXiv:2002.08836 [math.AG].
	
\bibitem[DoRo21]{DoRo21}
    A.-T. Do, S. Rollenske,
    \emph{Gorenstein stable surfaces with $K_X^2=1$ and $\chi(\O_X)=2$},
    arXiv:2109.11966 [math.AG].

\bibitem[D96]{D96}
	I. Dolgachev,
	\emph{Mirror symmetry for lattice polarized K3 surfaces},
	Algebraic geometry, 4. J. Math. Sci. 81 (1996), no. 3, 2599--2630.

\bibitem[EV92]{EV92}
	H. Esnault, E. Viehweg,
	\emph{Lectures on vanishing theorems},
	DMV Seminar, 20. Birkh\"auser Verlag, Basel, 1992.

\bibitem[ES20]{ES20}
    J. Evans, I. Smith,
    \emph{Bounds on Wahl singularities from symplectic topology},
    Algebr. Geom. 7 (2020), no.1, 59--85.

\bibitem[FPR17]{FPR17}
	M. Franciosi, R. Pardini, S. Rollenske,
	\emph{Gorenstein stable surfaces with $K_X^2=1$ and $p_g>0$},
	Math. Nachr. 290 (2017), no.5-6, 794--814.

\bibitem[H12]{H12}
    P. Hacking,
    \emph{Compact moduli spaces of surfaces of general type}.
    Compact moduli spaces and vector bundles, 1--18, Contemp. Math., 564, Amer. Math. Soc., Providence, RI, 2012.

\bibitem[HTU17]{HTU13}
    P. Hacking, J. Tevelev, G. Urz\'ua,
    \emph{Flipping surfaces},
    J. Algebraic Geom. 26 (2017), no. 2, 279–-345.

\bibitem[L03]{L03}
	A. Langer,
	\emph{Logarithmic orbifold Euler numbers of surfaces with applications},
	Proc. London Math. Soc. (3) 86(2003), no. 2, 358--396.

\bibitem[LP07]{LP07}
	Y. Lee, J. Park,
	\emph{A simply connected surface of general type with $p_g=0$ and $K^2=2$}, 		Invent. Math. 170 (2007), 483--505.

\bibitem[K92]{K92}
    Y. Kawamata,
    \emph{Moderate degenerations of algebraic surfaces},
    Complex algebraic varieties (Bayreuth, 1990), 113--132, Lecture Notes in Math., 1507, Springer, Berlin, 1992.

\bibitem[KSB88]{KSB88}
     J. Koll\'ar, N. I. Shepherd-Barron,
     \emph{Threefolds and deformations of surface singularities},
     Invent. Math. 91 (1988), 299--338.

\bibitem[Ky77]{Ky77}
    V. I. Kynef, 
    \emph{An example of a simply connected surface of general type for which the local Torelli theorem does not hold}(Russian),
    C.R. Acad. Bulgare Sc. 30, n. 3 (1977) pp. 323--325.

\bibitem[M03]{M03}
     M. Murakami,
	\emph{Minimal algebraic surfaces of general type with $c_1^2=3$, $p_g=1$ and $q=0$, which have non-trivial 3-torsion divisors},
	J. Math. Kyoto Univ. 43(2003), no. 1, 203--215.

\bibitem[M61]{M61}
     D. Mumford,
     \emph{The topology of normal singularities of an algebraic surface and a criterion for simplicity},
     Inst. Hautes \'Etudes Sci. Publ. Math. No. 9 (1961), 5--22.

\bibitem[PPS09a]{PPS09a}
     H. Park, J. Park, D. Shin,
     \emph{A simply connected surface of general type with $p_g=0$ and $K^2=3$},
     Geom. Topol. 13 (2009), no. 2, 743--767.

\bibitem[PPS09b]{PPS09b}
     H. Park, J. Park, D. Shin,
     \emph{A simply connected surface of general type with $p_g=0$ and $K^2=4$},
     Geom. Topol. 13 (2009), no. 3, 1483--1494.

\bibitem[PPS13]{PPS13}
     H. Park, J. Park, D. Shin,
     \emph{Surfaces of general type with $p_g=1$ and $q=0$},
     J. Korean Math. Soc. 50 (2013), no. 3, 493--507.

\bibitem[RTU17]{RTU17}
	J. Rana, J. Tevelev, G. Urz\'ua,
	\emph{The Craighero-Gattazzo surface is simply connected},
	Compos. Math. 153 (2017), no. 3, 557--585.

\bibitem[RU19]{RU19}
    J. Rana, G. Urz\'ua,
    \emph{Optimal bounds for T-singularities in stable surfaces},
    Adv. Math. 345(2019), 814--844.
    
\bibitem[Ro20]{Ro20}
    X. Roulleau,
    \emph{An atlas of K3 surfaces with finite automorphism group},
    arXiv:2003.08985 [math.AG].

\bibitem[SZ01]{SZ01}
    I. Shimada, D.-Q. Zhang,
    \emph{Classification of extremal elliptic K3 surfaces and fundamental groups of open K3 surfaces},
    Nagoya Math. J., Vol. 161 (2001), 23--54.

\bibitem[SU16]{SU16}
    A. Stern, G. Urz\'ua,
    \emph{KSBA surfaces with elliptic quotient singularities, $\pi_1=1$, $p_g=0$, and $K^2=1,2$},
    Israel J. Math. 214 (2016), no. 2, 651--673.

\bibitem[TY87]{TY87}
	G. Tian, S.-T. Yau,
	\emph{Existence of K\"ahler-Einstein metrics on complete K\"ahler manifolds and their applications to algebraic geometry},
	Mathematical aspects of string theory (San Diego, Calif., 1986), 574--628, Adv. Ser. Math. Phys., 1, World Sci. Publishing, Singapore, 1987.

\bibitem[T80]{T80}
	A. N. Todorov,
	\emph{Surfaces of general type with $p_g=1$ and $(K,K)=1$. I.},
	Ann. Sci. \'Ecole Norm. Sup. (4) 13(1980), no.1, 1--21.

\bibitem[T81]{T81}
	A. T. Todorov,
	\emph{A construction of surfaces with $p_g=1$, $q=0$ and $2\leq K^2 \leq 8$. Counterexamples of the global Torelli theorem.},
	Invent. Math. 63 (1981), no. 2, 287--304.

\bibitem[U16a]{U16a}
     G. Urz\'ua,
     \emph{$\Q$-Gorenstein smoothings of surfaces and degenerations of curves},
     Rend. Semin. Mat. Univ. Padova 136 (2016), 111--136.

\bibitem[U16b]{U16b}
	G. Urz\'ua,
	\emph{Identifying neighbors of stable surfaces},
	Ann. Sc. Norm. Super. Pisa Cl. Sci. (5) 16 (2016), no. 4, 1093--1122.

\bibitem[UV21]{UV21}
	G. Urz\'ua, N. Vilches,
	\emph{On wormholes in the moduli space of surfaces},
	arXiv:2102.02177 [math.AG], to appear in Algebraic Geometry.

\bibitem[Wahl81]{Wahl81}
    J. Wahl.
    \emph{Smoothings of normal surface singularities},
    Topology 20(1981), no. 3, 219--246.

\bibitem[Ye99]{Ye99}
    Q. Ye,
    \emph{On extremal elliptic K3 surfaces},
    J. Korean Math. Soc. 36 (1999), No. 6, 1091--1113.

\bibitem[X91]{X91}
    G. Xiao,
    \emph{$\pi_1$ of elliptic and hyperelliptic surfaces},
    Internat. J. Math. 2 (1991), no. 5, 599--615.

\end{thebibliography}
\end{document}